\documentclass{amsart}[12pt]
\usepackage{amsthm, bbold, amssymb}
\DeclareMathOperator\Gal{Gal}

\DeclareMathOperator\Pic{Pic}
\DeclareMathOperator\Aut{Aut}

\DeclareMathOperator{\NS}{\mathbf{NS}}

\DeclareMathOperator\Hom{Hom}
\DeclareMathOperator\End{End}

\usepackage{color}
\usepackage[dvipsnames]{xcolor}
\usepackage{mathrsfs} 
\usepackage{algorithm}
\usepackage{algorithmic}
\usepackage{tikz}
\usepackage{enumitem}
\usepackage[multiple]{footmisc}

\newcommand{\authnote}[2][]{\noindent {\if!#1!  {\bf TODO} \else {\small \bf #1} \fi: #2}}

\newcommand{\bennote}[1]{{\authnote[Ben]{\color{Purple} \textbf{\small #1}}}}

\newcommand{\isom}{\cong}
\newcommand{\ra}{\rightarrow}

\def\1{\mathbb{1}}

\def\cL{\mathcal{L}}

\def\cO{\mathcal{O}}

\def\cM{\mathcal{M}}

\def\lO{\mathfrak{o}} 
\def\ll{{(\ell,\ell)}}
\def\pr{\mathrm{pr}}
\def\Cl{\mathrm{Cl}}

\def\ff{\mathfrak{f}}

\def\A{\mathbb{A}}
\def\F{\mathbb{F}}
\def\Z{\mathbb{Z}}

\def\Q{\mathbb{Q}}

\def\GL{\text{GL}}
\def\SL{\text{SL}}

\theoremstyle{plain} 

\newtheorem{theorem}{Theorem}
\newtheorem{maincounter}{maincounter}[section]
\newtheorem{proposition}[maincounter]{Proposition}
\newtheorem{lemma}[maincounter]{Lemma}

\newtheorem{definition}[maincounter]{Definition}
\newtheorem{notation}[maincounter]{Notation}

\theoremstyle{remark}
\newtheorem{remark}[maincounter]{Remark}
\newtheorem{example}[maincounter]{Example}

\usepackage[all]{xy}
\usepackage{tikz-cd}

\title{Isogeny graphs of ordinary abelian varieties}
\author[E. Hunter Brooks, Dimitar Jetchev]{Ernest Hunter Brooks, Dimitar Jetchev}
\address{\'Ecole Polytechnique F\'ed\'erale de Lausanne, EPFL SB MATHGEOM GR-JET, Switzerland}
\email{ernest.brooks@epfl.ch}
\email{dimitar.jetchev@epfl.ch}
\author{Benjamin Wesolowski}
\address{\'Ecole Polytechnique F\'ed\'erale de Lausanne, EPFL IC LACAL, Switzerland}
\email{benjamin.wesolowski@epfl.ch}

\begin{document}

\date{\today.}


\begin{abstract}
Fix a prime number $\ell$. Graphs of isogenies of degree a power of $\ell$ are well-understood for elliptic curves, but not for higher-dimensional abelian varieties. We study the case of absolutely simple ordinary abelian varieties over a finite field. We analyse graphs of so-called $\mathfrak l$-isogenies, resolving that they are (almost) volcanoes in any dimension. Specializing to the case of principally polarizable abelian surfaces, we then exploit this structure to describe graphs of a particular class of isogenies known as \emph{$(\ell, \ell)$}-isogenies: those whose kernels are maximal isotropic subgroups of the $\ell$-torsion for the Weil pairing.
We use these two results to write an algorithm giving a path of computable isogenies from an arbitrary absolutely simple ordinary abelian surface towards one with maximal endomorphism ring, which has immediate consequences for the CM-method in genus 2, for computing explicit isogenies, and for the random self-reducibility of the discrete logarithm problem in genus 2 cryptography.
\end{abstract}
\maketitle

\section{Introduction}

\subsection{Background}
Graphs of isogenies of principally polarized abelian varieties of dimension $g$ have been an extensive object of study in both number theory and mathematical cryptology. When $g = 1$, Kohel \cite{kohel:thesis} gave a description of the structure of such graphs and used it to compute the endomorphism ring of an elliptic curve over a finite field. This description has subsequently been utilized in a variety of cryptographic applications such as point counting on elliptic curves \cite{fouquet-morain}, random self-reducibility of the elliptic curve discrete logarithm problem in isogeny classes \cite{jmv:asiacrypt, jmv:jnt}, generating elliptic curves with a prescribed number of points via the CM method based on the Chinese Remainder Theorem~\cite{sutherland}, as well as computing modular polynomials \cite{broker-lauter-sutherland}. 

When $g > 1$, the problem of describing the structure of these graphs becomes harder. The literature has seen a number of attempts to generalize Kohel's thesis, yet the structure of these isogeny graphs has not been studied systematically. For $g = 2$, Br\"oker, Gruenewald and Lauter~\cite{broker-gruenewald-lauter} proved that graphs of $\ll$-isogenies of abelian surfaces are not volcanoes.

In~\cite{lauter-robert}, Lauter and Robert observed that from a random abelian surface, it might not always be possible to reach an isogenous one with maximal endomorphism ring (locally at $\ell$) using only $\ll$-isogenies. 
Following the footsteps of Kohel, Bisson~\cite[Ch.5]{bisson} sketched the relation between isogeny graphs and the lattice of orders in the endomorphism algebra for abelian varieties of higher dimension. This provides a first approximation of the global structure of the graphs, but allows no fine-grained analysis.

It was also unclear whether the notion of $\ll$-isogenies is the right one to generalize the structure of isogeny graphs.
Ionica and Thom\'e \cite{ionica-thome} observed that the subgraph of $\ll$-isogenies restricted to surfaces with maximal real order in $K_0$ (globally) could be studied through what they called $\mathfrak l$-isogenies, where $\mathfrak l$ is a prime ideal in $K_0$ above $\ell$. They suggest that the $\mathfrak l$-isogeny graphs should be volcanoes, under certain assumptions\footnote{The proof of \cite[Prop.15]{ionica-thome} gives a count of the number of points at each level of the graph, but does not allow a conclusive statement on the edge structure, and thus does not appear to prove that the graph is a volcano.}.
When $\mathfrak l$ is principal, of prime norm, generated by a real, totally positive endomorphism $\beta$, then $\mathfrak l$-isogenies coincide with the cyclic $\beta$-isogenies from \cite{dudeanu-jetchev-robert} --- an important notion, since these are the cyclic isogenies preserving principal polarizability. 

Our main contributions include a full description of graphs of $\mathfrak l$-isogenies for any $g \geq 1$. This proves the claims of~\cite{ionica-thome} and extends them to a much more general setting. For $g = 2$, we exploit this $\mathfrak l$-structure to provide a complete description of graphs of $\ll$-isogenies preserving the maximal real multiplication locally at~$\ell$. We also explore the structure of $\ll$-isogenies when the real multiplication is not necessarily locally maximal. As an application of these results, we build an algorithm that, given as input a principally polarized abelian surface, finds a path of computable isogenies leading to a surface with maximal endomorphism ring. This was a missing --- yet crutial --- building block for the CRT-based CM-method in dimension 2, for computing explicit isogenies between two given surfaces, and for the random self-reducibility of the discrete logarithm problem in genus 2 cryptography. Applications are discussed more thoroughly in Section~\ref{subsec:introApplications}.

This structure of $\mathfrak l$-isogenies, when one assumes that $\mathfrak{l}$ is of prime norm and trivial in the narrow class group of $K_0$, implies in particular that graphs of cyclic $\beta$-isogenies are volcanoes. 
In parallel to the present work, Chloe Martindale has recently announced a similar result on cyclic $\beta$-isogenies.
It will be found in her forthcoming Ph.D. thesis, as part of a larger project aimed at computing Hilbert class polynomials and modular polynomials in genus 2. Her results, which are proven in a complex-analytic setting different from our $\ell$-adic methods, yield the same description of the graph in this particular case.

\subsection{Setting}\label{subsec:setting}
For a given ordinary, absolutely simple abelian variety $\mathscr A$ over a finite field $k = \F_q$, the associated endomorphism algebra $\End (\mathscr A) \otimes_\Z \Q$ is isomorphic to a CM-field $K$, i.e., a totally imaginary quadratic extension of a totally real number field $K_0$. Moreover, the dimension $g$ of $\mathscr A$ equals the degree $[K_0 : \Q$]. The endomorphism ring $\End(\mathscr A)$ identifies with an order $\cO$ in $K$. The Frobenius endomorphism $\pi$ of $\mathscr A$ generates the endomorphism algebra $K = \Q(\pi)$, and its characteristic polynomial determines its $k$-isogeny class, by Tate's isogeny theorem~\cite{Tate1966}. In particular, since $\End_k(\mathscr A) = \End_{\overline k}(\mathscr A)$ (see~\cite[Thm.7.2.]{Waterhouse1969}), all isogenous varieties (over $\overline k$) share the same CM-field $K$, and their endomorphism rings all correspond to orders in $K$. Thus, the structure of isogeny graphs is related to the structure of the lattice of orders of the field $K$.

The choice of an isomorphism $\End(\mathscr A) \otimes_\Z \Q \cong K$ naturally induces an embedding $\imath_\mathscr B : \End(\mathscr B) \ra K$ for any variety $\mathscr B$ that is isogenous to $\mathscr A$, and it does not depend on the choice of an isogeny. We can then unambiguously denote by $\cO(\mathscr B)$ the order in $K$ corresponding to the endomorphism ring of any $\mathscr B$. 
Define the suborder $\cO_0(\mathscr A) = \cO(\mathscr A) \cap K_0$; the variety $\mathscr A$ is said to have \emph{real multiplication} (RM) by $\cO_0(\mathscr A)$. Recall the conductor $\ff$ of an order $\cO$ in a number field $L$ is defined as
$$\ff = \{x \in L\ |\ x\cO_L \subseteq \cO\}.$$
Equivalently, it is the largest subset of $L$ which is an ideal in both $\cO_L$ and $\cO$.

Fix once and for all a prime number $\ell$ different from the characteristic of the finite field $k$, and write $\lO(\mathscr A) = \cO(\mathscr A) \otimes_{\Z} \Z_\ell$, the \emph{local order} of $\mathscr A$. It is an order in the algebra $K_\ell = K \otimes_{\Q} \Q_\ell$. Also, $\lO_K = \cO_K \otimes_{\Z} \Z_\ell$ is the maximal order in $K_\ell$.
Finally, write $\lO_0(\mathscr  A)$ for the \emph{local real order} $\cO_0(\mathscr  A) \otimes_{\Z} \Z_\ell$, which is an order in the algebra $K_{0,\ell} = K_0 \otimes_{\Q} \Q_\ell$, and let $\lO_0 = \cO_{K_0} \otimes_{\Z} \Z_\ell$.

\subsection{Main results}
When $\mathscr A$ is an elliptic curve, the lattice of orders is simple: $K$ being a quadratic number field (i.e. $K_0 = \Q$), all the orders in $K$ are of the form $\cO_c = \Z + c\cO_K$, with $c \in \Z$ generating the conductor of $\mathcal{O}_c$. Locally at a prime number $\ell$, the lattice of orders in $K_\ell$ is simply the chain $\lO_{K} \supset \Z_\ell + \ell \lO_{K} \supset \Z_{\ell} + \ell^2 \lO_{K} \supset...$.
The (local) structure of the lattice of orders of a CM-field $K$ is in general not as simple as the linear structure arising in the case of an imaginary quadratic field. This constitutes the main difficulty in generalizing the structural results to $g > 1$.
For the rest of the paper, we let $g > 1$, and fix an isogeny class whose endomorphism algebra is the CM-field $K$.

\subsubsection{Isogeny graphs preserving the real multiplication} \label{subsubsec:theorem1}
In the case of quadratic number fields, the inclusion of orders corresponds to the divisibility relation of conductors.  Neither the one-to-one correspondence between orders and conductors, nor the relationship between inclusion and divisibility holds in higher degree. 
We can, however, prove that such a correspondence between orders and conductors, and inclusion and divisibility still holds if we restrict to orders with maximal real multiplication, i.e., $\cO_{K_0} \subset \cO$. More than that, it even holds locally, i.e., for the orders of $K_\ell$ containing $\lO_0$. More precisely, we show in Section~\ref{sec:orders}, Theorem~\ref{thm:classificationOrdersMaxRM}, that any order in $K$ (respectively $K_\ell$) with maximal real multiplication is of the form $\cO_{K_0} + \mathfrak f \cO_K$ (respectively $\lO_{0} + \mathfrak f \lO_K$) for some ideal $\mathfrak f$ in $\cO_{K_0}$. Our first results use this classification to provide a complete description of graphs of isogenies preserving the maximal real multiplication locally at $\ell$.
The main building block for isogenies preserving the real multiplication is the notion of $\mathfrak l$-isogeny.

\begin{definition}[$\mathfrak l$-isogeny]\label{def:frakLIso}
Let $\mathfrak l$ be a prime above $\ell$ in $K_0$, and $\mathscr A$ a variety in the fixed isogeny class. Suppose $\mathfrak l$ is coprime to the conductor of $\cO_0(\mathscr A)$.
An $\mathfrak l$-isogeny from $\mathscr A$ is an isogeny whose kernel is a proper, $\cO_0(\mathscr A)$-stable subgroup of\footnote{By abuse of notation, we write $\mathscr A[\mathfrak l]$ in place of $\mathscr A[\mathfrak l \cap \cO(\mathscr A)]$.} $\mathscr A[\mathfrak l]$.
\end{definition}

\begin{remark}
The degree of an $\mathfrak l$-isogeny is $N\mathfrak l$. 
\end{remark}

We will therefore study the structure of the graph $\mathscr W_\mathfrak l$ whose vertices are the isomorphism classes of abelian varieties $\mathscr A$ in the fixed isogeny class, which have maximal real multiplication locally at $\ell$ (i.e., $\lO_0 \subset \lO(\mathscr A)$), and there is an edge of multiplicity $m$ from such a vertex with representative $\mathscr A$ to a vertex $\mathscr B$ if there are $m$ distinct subgroups $\kappa\subset \mathscr A$ that are kernels of $\mathfrak l$-isogenies such that $\mathscr A/\kappa \cong \mathscr B$ (of course, the multiplicity $m$ does not depend on the choice of the representative $\mathscr A$).

\begin{remark}
When $\mathfrak l$ is trivial in the narrow class group of $K_0$, then $\mathfrak l$-isogenies preserve principal polarizability.
The graph $\mathscr W_\mathfrak l$ does not account for polarizations, but it is actually easy to add polarizations back to graphs of unpolarized varieties, as will be discussed in Section~\ref{sec:Polarizations}.
\end{remark}

Each vertex $\mathscr A$ of this graph $\mathscr W_\mathfrak l$ has a level, given by the valuation $v_\mathfrak l(\mathscr A)$ at~$\mathfrak l$ of the conductor of $\cO(\mathscr A)$.
Our first result, Theorem~\ref{thm:lisogenyvolcanoes}, completely describes the structure of the connected components of $\mathscr W_\mathfrak l$, which turns out to be closely related to the volcanoes observed for cyclic isogenies of elliptic curves. It is proven in Subsection~\ref{subsec:frakLIsogenyGraphs}.

\begin{theorem}\label{thm:lisogenyvolcanoes}
Let $\mathscr V$ be any connected component of the leveled $\mathfrak l$-isogeny graph $(\mathscr W_\mathfrak l, v_\mathfrak l)$.
For each $i \geq 0$, let $\mathscr V_i$ be the subgraph of $\mathscr V$ at level $i$.
We have:
\begin{enumerate}[label=(\roman*)]
\item For each $i \geq 0$, the varieties in $\mathscr V_i$ share a common endomorphism ring $\cO_i$. The order $\cO_0$ can be any order with locally maximal real multiplication at $\ell$, whose conductor is not divisible by $\mathfrak l$;
\item The level $\mathscr V_0$ is isomorphic to the Cayley graph of the subgroup of $\Pic(\cO_0)$ with generators the prime ideals above $\mathfrak l$; fixing $\mathscr A \in \mathscr V_0$, an isomorphism is given by sending any ideal class $[\mathfrak a]$ to the isomorphism class of $\mathscr A/\mathscr A[\mathfrak a]$;
\item For any $\mathscr A \in \mathscr V_0$, there are $\left(N(\mathfrak l)-\left(\frac{K}{\mathfrak l}\right)\right)/[\cO_{0}^\times : \cO_{1}^\times]$ edges of multiplicity $[\cO_{0}^\times : \cO_{1}^\times]$ from $\mathscr A$ to distinct vertices of~$\mathscr V_{1}$ (where $\left(\frac{K}{\mathfrak l}\right)$ is $-1$, $0$ or $1$ if $\mathfrak l$ is inert, ramified, or split in $K$);
\item For each $i > 0$, and any $\mathscr A \in \mathscr V_i$, there is one simple edge from $\mathscr A$ to a vertex of $\mathscr V_{i-1}$, and $N(\mathfrak l)/[\cO_{i}^\times : \cO_{i+1}^\times]$ edges of multiplicity $[\cO_{i}^\times : \cO_{i+1}^\times]$ to distinct vertices of $\mathscr V_{i+1}$, and there is no other edge from $\mathscr A$;
\item For each path $\mathscr A \rightarrow \mathscr B \rightarrow \mathscr C$ where the first edge is descending, and the second ascending, we have $\mathscr C \cong \mathscr A / \mathscr A[\mathfrak l]$;
\item\label{ascendingImpliesDescending} For each ascending edge $\mathscr B \rightarrow \mathscr C$, there is a descending edge $\mathscr C \rightarrow \mathscr B / \mathscr B[\mathfrak l]$.
\end{enumerate}
In particular, the graph $\mathscr V$ is an $N(\mathfrak l)$-volcano if and only if $\cO_0^\times \subset K_0$ and $\mathfrak l$ is principal in ${\cO_0 \cap K_0}$.

Also, if $\mathscr V$ contains a variety defined over the finite field $k$, the subgraph containing only the varieties defined over $k$ consists of the subgraph of the first $v$ levels, where $v$ is the valuation at $\mathfrak l$ of the conductor of $\cO_{K_0}[\pi] = \cO_{K_0}[\pi, \pi^\dagger]$.
\end{theorem}

\subsubsection{Graphs of $(\ell, \ell)$-isogenies}
The following results focus on the case $g = 2$. In contrast to the case of elliptic curves, where a principal polarization always exists, the property of being principally polarizable is not even invariant under cyclic isogeny in genus $2$. In addition, basic algorithms for computing isogenies of elliptic curves from a given kernel (such as V\'elu's formulae \cite{velu}) are difficult to generalize and the only known methods \cite{robert,cosset-robert,lubicz-robert,dudeanu-jetchev-robert} assume certain hypotheses and thus do not apply for general isogenies of cyclic kernels. On the other hand, $\ll$-isogenies always preserve principal polarizability, and are computable with the most efficient of these algorithms~\cite{cosset-robert}. These $\ll$-isogenies are therefore an important notion, and we are interested in understanding the structure of the underlying graphs.

\begin{definition}[$\ll$-isogeny]\label{def:ll}
Let $(\mathscr A, \xi_{\mathscr{A}})$ be a principally polarized abelian surface. We call an isogeny $\varphi \colon \mathscr A \rightarrow \mathscr B$ an $\ll$-isogeny (with respect to $\xi_{\mathscr{A}}$) if $\ker (\varphi)$ is a maximal isotropic subgroup of $\mathscr A[\ell]$ with respect to the Weil pairing on $\mathscr A[\ell]$ induced by the polarization isomorphism corresponding to $\xi_\mathscr A$. 
\end{definition}

One knows that if $\varphi \colon \mathscr A \ra \mathscr B$ is an $(\ell, \ell)$-isogeny, then there is a unique principal polarization $\xi_{\mathscr{B}}$ on $\mathscr B$ such that $\varphi^* \xi_{\mathscr{B}} = \xi_{\mathscr{A}}^{\ell}$ (this is a consequence of Grothendieck descent \cite[pp.290--291]{mumford:eq1}; see also \cite[Prop. 2.4.7]{drobert:thesis}). This allows us to view an isogeny of \emph{a priori} non-polarized abelian varieties $\varphi$ as an isogeny of polarized abelian varieties $\varphi \colon (\mathscr A, \cL^\ell) \ra (\mathscr B, \cM)$.\\

First, we restrict our attention to abelian surfaces with maximal real multiplication at $\ell$. The description of $\mathfrak l$-isogeny graphs provided by Theorem~\ref{thm:lisogenyvolcanoes} leads to a complete understanding of graphs of $\ll$-isogenies preserving the maximal real order locally at $\ell$, via the next theorem.
More precisely, we study the structure of the graph $\mathscr G_{\ell,\ell}$ whose vertices are the isomorphism classes of principally polarizable surfaces $\mathscr A$ in the fixed isogeny class, which have maximal real multiplication locally at $\ell$ (i.e., $\lO_0 \subset \lO(\mathscr A)$), with an edge of multiplicity $m$ from such a vertex $\mathscr A$ to a vertex $\mathscr B$ if there are $m$ distinct subgroups $\kappa\subset \mathscr A$ that are kernels of $\ll$-isogenies such that $\mathscr A/\kappa \cong \mathscr B$. This definition will be justified by the fact that the kernels of $\ll$-isogenies preserving the maximal real multiplication locally at $\ell$ do not depend on the choice of a principal polarization on the source (see Remark~\ref{rem:RMPreservingEllEllDoNotDependOnPol}).
The following theorem is proven in Subsection~\ref{subsec:ellellMaxRM}, where its consequences are discussed in details.

\begin{theorem}\label{thm:ellellLCombinations}
Suppose that $\mathscr A$ has maximal real multiplication locally at $\ell$.
Let $\xi$ be any principal polarization on $\mathscr A$.
There is a total of $\ell^3 + \ell^2 + \ell + 1$ kernels of $\ll$-isogenies from~$\mathscr A$ with respect to $\xi$.
Among these, the kernels whose target also has maximal local real order do not depend on $\xi$, and are:
\begin{enumerate}[label=(\roman*)]
\item the $\ell^2+1$ kernels of $\ell\cO_{K_0}$-isogenies if $\ell$ is inert in $K_0$,
\item the $\ell^2 + 2\ell + 1$ kernels of compositions of an $\mathfrak l_1$-isogeny with an $\mathfrak l_2$-isogeny if $\ell$ splits as $\mathfrak l_1\mathfrak l_2$ in $K_0$,
\item the $\ell^2 + \ell + 1$ kernels of compositions of two $\mathfrak l$-isogenies if $\ell$ ramifies as $\mathfrak l^2$ in $K_0$.
\end{enumerate}
The other $\ll$-isogenies have targets with real multiplication by $\lO_1 = \Z_\ell + \ell \lO_0$.
\end{theorem}

Second, we look at $\ll$-isogenies when the real multiplication is not maximal at $\ell$.
Note that since $g = 2$, even though the lattice of orders in $K$ is much more intricate than in the quadratic case, there still is some linearity when looking at the suborders $\cO_0(\mathscr A) = \cO(\mathscr A) \cap K_0$, since $K_0$ is a quadratic number field. 
For any variety $\mathscr A$ in the fixed isogeny class, there is an integer $f$, the conductor of $\cO_0(\mathscr A)$, such that 
$\cO_0(\mathscr A) = \Z + f\cO_{K_0}$.
The local order $\lO_0(\mathscr  A)$
is exactly the order $\lO_n = \Z_\ell + \ell^n\lO_0$ in $K_{0,\ell}$, where $n  =  v_\ell(f)$ is the valuation of $f$ at the prime $\ell$.

The next result describes how $(\ell, \ell)$-isogenies can navigate between these ``levels'' of real multiplication. 
Let $\varphi \colon \mathscr A \ra \mathscr B$ be an $\ll$-isogeny with respect to a polarization $\xi$ on $\mathscr A$. If $\lO_0(\mathscr A) \subset \lO_0(\mathscr B)$, we refer to $\varphi$ as an \emph{RM-ascending} isogeny; if $\lO_0(\mathscr B) \subset \lO_0(\mathscr A)$, we call $\varphi$ \emph{RM-descending}; otherwise, if $\lO_0(\mathscr A) = \lO_0(\mathscr B)$, $\varphi$ is called \emph{RM-horizontal}.  
Note that we start by considering $\ll$-isogenies defined over the algebraic closure of the finite field $k$; in virtue of Remark~\ref{rem:fieldOfDefIsogenies}, it is then easy to deduce the results on isogenies defined over $k$. The following assumes $n > 0$ since the case $n = 0$ is taken care of by Theorem~\ref{thm:ellellLCombinations}.

\begin{theorem}\label{RMupIso}
Suppose $\lO_0(\mathscr A) = \lO_n$ with $n > 0$. For any principal polarization $\xi$ on $\mathscr A$, the kernels of $\ll$-isogenies from $(\mathscr A, \xi)$ are:
\begin{enumerate}[label=(\roman*)]
\item A unique RM-ascending one, 
whose target has local order $\lO_{n-1}\cdot\lO(\mathscr A)$ (in particular, its local real order is $\lO_{n-1}$, and the kernel is defined over the same field as $\mathscr A$),
\item $\ell^2 + \ell$ RM-horizontal ones, and
\item $\ell^3$ RM-descending isogenies, whose targets have local real order $\lO_{n+1}$.
\end{enumerate}
\end{theorem}

The proof of this theorem is the matter of Section~\ref{sec:levelsRM}.

\subsection{Lattices in $\ell$-adic symplectic spaces}
The theorems stated above are proven using a different approach from the currently available analyses of the structure of $\ell$-power isogeny graphs. Rather than working with complex tori, we attach to an $\ell$-isogeny of abelian varieties a pair of lattices in an $\ell$-adic symplectic space, whose relative position is determined by the kernel of the isogeny, following the proof of Tate's isogeny theorem \cite{Tate1966}.

Inspired by ~\cite[\textsection 6]{Cornut04}, where the theory of Hecke operators on $\GL_2$ is used to understand the CM elliptic curves isogenous to a fixed curve, we analyze the possible local endomorphism rings (in $K_\ell = K \otimes_{\Q} \Q_\ell$) for an analogous notion of ``neighboring'' lattices. This method gives a precise count of the horizontal isogenies as well as the vertical isogenies that increase or decrease the local endomorphism ring at $\ell$.

\subsection{``Going up'' and applications}\label{subsec:introApplications}
One of the applications of the above structural results is an algorithm that, given as input a principally polarized abelian surface, finds a path of computable isogenies leading to a surface with maximal endomorphism ring, when it is possible (and we can charaterize when it is possible). This algorithm is built and analysed in Section~\ref{sec:goingUp}.

Such an algorithm has various applications. One of them is in generating (hyperelliptic) curves of genus 2 over finite fields with 
suitable security parameters via the CM method. The method is based on first computing invariants for the curve 
(Igusa invariants) and then using a method of Mestre \cite{mestre:genus2} (see also \cite{cardona-quer}) to reconstruct the equation of the curve. 
The computation of the invariants is expensive and there are three different ways to compute their minimal polynomials (the Igusa class polynomials): 1) complex analytic techniques \cite{vanwamelen:genus2, weng:genus2, Streng10}; 2) $p$-adic lifting techniques \cite{carls-kohel-lubicz, carls-lubicz, ghkr:2adic}; 3) a technique based on the Chinese Remainder Theorem \cite{eisentraeger-lauter, freeman-lauter, broker-gruenewald-lauter} (the \emph{CRT method}).

Even if 3) is currently the least efficient method, it is also the least understood and deserves more attention: its analog for elliptic curves holds the records for time and space complexity and for the size of the computed examples \cite{enge-sutherland, sutherland:crt}. 

The CRT method of \cite{broker-gruenewald-lauter} requires one to find an ordinary abelian surface $\mathscr A / \F_q$ whose endomorphism ring is the maximal order $\cO_K$ of the quartic CM field $K$ isomorphic to the 
endomorphism algebra $\End(\mathscr A) \otimes_{\Z} \Q$. This is obtained by trying random hyperelliptic curves $\mathscr C$ in the isogeny class and using the maximal endomorphism test of \cite{freeman-lauter}, thus making the algorithm quite inefficient. 

In \cite{lauter-robert}, the authors propose a different method based on $(\ell, \ell)$-isogenies that does not require the endomorphism ring to be maximal (generalizing the method of Sutherland \cite{sutherland:crt} for elliptic curves). Starting from an arbitrary abelian surface in the isogeny class, the method is based on a probabilistic algorithm for ``going up" to an abelian surface with maximal endomorphism ring. Although the authors cannot prove that the going-up algorithm succeeds with any fixed probability, the improvement is practical and heuristically, it reduces the running time of the CRT method in genus 2 from $\cO(q^3)$ to $\cO(q^{3/2})$. Our algorithm for going up takes inspiration from~\cite{lauter-robert}, but exploits our new structural results on isogeny graphs.

A second application is in the computation of an explicit isogeny between any two given principally polarized abelian surfaces in the same isogeny class. An algorithm is given in~\cite{JW15} to find an isogeny between two such surfaces with maximal endomorphism ring. 
This can be extended to other pairs of isogenous principally polarized abelian surfaces, by first computing paths of isogenies to reach the maximal endomorphism ring, then applying the method of~\cite{JW15}. 

Similarly, this ``going up'' algorithm can also extend results about the random self-reducibility of the discrete logarithm problem in genus 2 cryptography. The results of~\cite{JW15} imply that if the discrete logarithm problem is efficiently solvable on a non-negligible proportion of the Jacobians with maximal endomorphism ring within an isogeny class, then it is efficiently solvable for all isogenous Jacobians with maximal endomorphism ring.
For this to hold on any other Jacobian in the isogeny class, it only remains to compute a path of isogenies reaching the level of the maximal endomorphism ring.

Finally, we note that the ``going-up" algorithm can also be applied in the computation of endomorphism rings of abelian surfaces over finite fields, thus extending the work of Bisson \cite{bisson}. This will be the subject of a forthcoming paper.

\section{Orders} \label{sec:orders}

\subsection{Global and local orders}\label{subsec:localglobalorders}
An order in a number field is a full rank $\Z$-lattice which is also a subring. If  $\ell$ is a prime, and $L$ is a finite extension of $\Q_\ell$ or a finite product of finite extensions of $\Q_\ell$, an order in $L$ is a full rank $\Z_\ell$-lattice which is also a subring. If $K$ is a number field, write $K_\ell = K \otimes_\Q \Q_\ell$. In this section, if $\cO$ an order in $K$, write $\cO_\ell = \cO \otimes_\Z \Z_\ell$; then $\cO_\ell$ is an order in $K_\ell$. 

\begin{lemma}
Given a number field $K$ and a sequence $R(\ell)$ of orders in $K_\ell$, such that $R(\ell)$ is the maximal order in $K_\ell$ for almost all $\ell$, there exists a unique order $\cO$ in $K$ such that $\cO_\ell = R(\ell)$ for all $\ell$. This order $\cO$ is the intersection $\bigcap_{\ell} (R(\ell)\cap K)$.
\end{lemma}

\begin{proof}
This is well-known, but we include a proof for completeness. Let $n = [K:\Q]$ and pick a $\mathbb{Z}$-basis for the maximal order $\cO_K$ of $K$. With this choice, a lattice $\Lambda$ in $\cO_K$ may be described by a matrix in $M_n(\Z) \cap \GL_n(\Q)$ whose column vectors are a basis; this matrix is well-defined up to the left action of $\GL_n(\Z)$. Similarly, a local lattice $\Lambda_\ell$ in $K_\ell$ may be described by a matrix in $M_n(\Z_\ell) \cap \GL_n(\Q_\ell)$, well-defined up to the left-action of $\GL_n(\Z_\ell)$. It thus suffices to prove that, given matrices $M_\ell \in M_n(\Z_\ell) \cap \GL_n(\Q_\ell)$, almost all of which are in $\GL_n(\Z_\ell)$, there exists an $N \in M_n(\Z) \cap \GL_n(\Q)$ such that $NM_\ell^{-1} \in \GL_n(\Z_\ell)$. This follows from
$$
\GL_n(\A_\mathrm{fin}) = \GL_n(\Q) \cdot \prod_\ell \GL_n(\Z_\ell),
$$
a consequence of strong approximation for $\SL_n$ and the surjectivity of the determinant map $\GL_n(\Z_\ell) \to \Z_\ell^\times$ (see the argument in \cite[p. 52]{Gelbart}, which generalizes in an obvious way when $n > 2$). Finally, the identity $\cO = \bigcap_{\ell} (\cO_\ell\cap K)$ follows from the fact that $\tilde\cO = \bigcap_{\ell} (\cO_\ell\cap K)$ is an order in $K$ such that $\tilde\cO_\ell = \cO_\ell$ for all $\ell$.
\end{proof}

\subsection{Orders with maximal real multiplication}

Suppose that $K_0$ is a number field or finite product of extensions of $\Q_p$, and let $K$ a quadratic extension of $K_0$ (i.e., an algebra of the form $K_0[x]/f(x)$, where $f$ is a separable quadratic polynomial). The non-trivial element of $\Aut(K/K_0)$ will be denoted $\dagger$. In the case that $K$ is a CM-field and $K_0$ its maximally real subfield, Goren and Lauter~\cite{GorenLauter09} proved that if $K_0$ has a trivial class group, the orders with maximal real multiplication, i.e., the orders containing $\cO_{K_0}$, are characterized by their conductor --- under the assumption that ideals of $\cO_K$ fixed by $\Gal(K/K_0)$ are ideals of $\cO_{K_0}$ augmented to $\cO_K$, which is rather restrictive, since it implies that no finite prime of $K_0$ ramifies in $K$. In that case, these orders are exactly the orders $\cO_{K_0} + \ff_0 \cO_K$, for any ideal $\ff_0$ in $\cO_{K_0}$. We generalize this result to an arbitrary quadratic extension; abusing language, we will continue to say an order of $K$ has ``maximal real multiplication'' if it contains $\cO_{K_0}$.

\begin{theorem}\label{thm:classificationOrdersMaxRM}
The map $\ff_0 \mapsto \cO_{K_0} + \ff_0\cO_K$ is a bijection between the set of ideals in $\cO_{K_0}$ and the set of orders in $K$ containing $\cO_{K_0}$. More precisely,
\begin{enumerate}[label=(\roman*)]
\item \label{thmClassificationMaxRM1} for any ideal $\ff_0$ in $\cO_{K_0}$, the conductor of $\cO_{K_0} + \ff_0\cO_K$ is $\ff_0\cO_K$, and
\item \label{thmClassificationMaxRM2} for any order $\cO$ in $K$ with maximal real multiplication and conductor $\ff$, one has $\cO = \cO_{K_0} + (\ff \cap \cO_{K_0}) \cO_K$.
\end{enumerate}
\end{theorem}

\begin{lemma}\label{lemma_rosatiStable}
An order $\cO$ in $K$ is stable under $\dagger$ if and only if $\cO \cap K_0 = (\cO + \cO^\dagger) \cap K_0$.
\end{lemma}

\begin{proof}
The direct implication is obvious. For the other direction, suppose $\cO \cap K_0 = (\cO + \cO^\dagger) \cap K_0$ and let $x \in \cO$. Then, $x + x^\dagger \in (\cO + \cO^\dagger) \cap K_0 = \cO \cap K_0 \subset \cO$, which proves that $x^\dagger \in \cO$.
\end{proof}

\begin{lemma}\label{bijfrak}
Let $\ff$ and $\mathfrak g$ be two ideals in $\cO_{K}$, such that $\mathfrak g$ divides $\ff$. Let $\pi : \cO_K \rightarrow \cO_K/\ff$ be the natural projection. The canonical isomorphism between $(\cO_{K_0} + \ff) / \ff$ and $\cO_{K_0}/(\cO_{K_0} \cap \ff)$ induces a bijection between $\pi(\cO_{K_0}) \cap \pi(\mathfrak g)$ and $(\cO_{K_0} \cap \mathfrak g)/(\cO_{K_0} \cap \ff)$.
\end{lemma}

\begin{proof}
Any element in $ \pi(\cO_{K_0}) \cap \pi(\mathfrak g)$ can be written as $\pi(x) = \pi(y)$ for some $x \in  \cO_{K_0}$ and $y \in \mathfrak g$. Then, $x-y \in \ff \subset \mathfrak g$, so $x = (x-y) + y \in \mathfrak g$. So 
$$\pi(\mathfrak g) \cap \pi(\cO_{K_0}) = \pi(\mathfrak g \cap \cO_{K_0}) \cong (\mathfrak g \cap \cO_{K_0})/(\ff \cap \cO_{K_0}),$$
where the last relation comes from the canonical isomorphism between the rings $(\cO_{K_0} + \ff) / \ff$ and $\cO_{K_0}/(\cO_{K_0} \cap \ff)$.
\end{proof}

\begin{lemma}\label{maxRM_rosatiStable}
Let $\cO$ be an order in $K$ of conductor $\ff$ with maximal real multiplication. Then, $\cO$ is stable under $\dagger$ and $\ff$ comes from an ideal of $\cO_{K_0}$, i.e., $\ff = \ff_0\cO_{K}$, where $\ff_0$ is the $\cO_{K_0}$-ideal $\ff \cap \cO_{K}$.
\end{lemma}

\begin{proof}
From Lemma~\ref{lemma_rosatiStable}, it is obvious that any order with maximal real multiplication is stable under $\dagger$. Its conductor $\ff$ is thereby a $\dagger$-stable ideal of $\cO_{K}$. For any prime ideal $\mathfrak p_0$ in $\cO_{K_0}$, let $\ff_{\mathfrak p_0}$ be the part of the factorization of $\ff$ that consists in prime ideals above $\mathfrak p_0$. Then, $\ff = \prod_{\mathfrak p_0} \ff_{\mathfrak p_0}$, and each $\ff_{\mathfrak p_0}$ is $\dagger$-stable. 
It is easy to see that each  $\ff_{\mathfrak p_0}$ comes from an ideal of $\cO_{K_0}$ when $\mathfrak p_0$ is inert or splits in $\cO_K$.
Now suppose it ramifies as $\mathfrak p_0\cO_K = \mathfrak p^2$. Then $\ff_{\mathfrak p_0}$ is of the form $\mathfrak p^\alpha$. If $\alpha$ is even, $\ff_{\mathfrak p_0} = \mathfrak p_0^{\alpha/2}\cO_K$. We now need to prove that $\alpha$ cannot be odd.

By contradiction, suppose $\alpha = 2\beta + 1$ for some integer $\beta$. Let $\pi : \cO_K \rightarrow \cO_K / \ff$ be the canonical projection. The ring $\pi(\cO)$ contains $\pi(\cO_{K_0}) = (\cO_{K_0} + \ff)/\ff$. Write $\ff = \mathfrak p^\alpha \mathfrak g$ and let us prove that $\pi(\mathfrak p^{\alpha - 1}\mathfrak g) \subset \pi(\cO_{K_0})$.
From Lemma~\ref{bijfrak}, 
$$\left| \pi(\cO_{K_0}) \cap \pi(\mathfrak p^{\alpha - 1}\mathfrak g) \right| = |\mathfrak  p_0^{\beta}\mathfrak g_0/\mathfrak p_0^{\beta+1}\mathfrak g_0| = N(\mathfrak p_0) = N(\mathfrak p) = |\pi(\mathfrak p^{\alpha - 1}\mathfrak g)|,$$
where $N$ denotes the absolute norm, so $\pi(\mathfrak p^{\alpha - 1}\mathfrak g) \subset \pi(\cO_{K_0}) \subset \pi(\cO)$. Finally,
$$\mathfrak p^{\alpha - 1}\mathfrak g = \pi^{-1}(\pi(\mathfrak p^{\alpha - 1}\mathfrak g)) \subset \pi^{-1}(\pi(\cO)) = \cO,$$
which contradicts the fact that $\ff$ is the biggest ideal of $\cO_K$ contained in $\cO$.
\end{proof}

\begin{lemma}\label{lemmaModuleQuotient}
Let $\ff_0$ be an ideal in $\cO_{K_0}$, and $R = \cO_{K_0}/\ff_0$. There is an element $\alpha \in \cO_K$ such that $\cO_K/\ff_0\cO_{K} = R \oplus R \alpha$.
\end{lemma}

\begin{proof}
The order $\cO_K$ is a module over $\cO_{K_0}$. It is locally free, and finitely generated, thus it is projective. Since $\cO_{K_0}$ is a regular ring, the submodule $\cO_{K_0}$ in $\cO_{K}$ is a direct summand, i.e., there is an $\cO_{K_0}$-submodule $M$ of $\cO_K$ such that $\cO_{K}~=~\cO_{K_0}~\oplus~M$. Then,
$\cO_{K}/\ff_0\cO_{K} = R \oplus M/\ff_0M.$
Let $A$ be $\Z$ if $K$ is a number field and $\Z_p$ if it is a finite product of extensions of $\Q_p$. In the former case write $n$ for $[K:\Q]$ and in the latter for the dimension of $K_p$ as a $\Q_p$-vector space. As modules over $A$, $\cO_K$ is of rank $2n$ and $\cO_{K_0}$ of rank $n$, hence $M$ must be of rank $n$. Therefore, as an $\cO_{K_0}$-module, $M$ is isomorphic to an ideal $\mathfrak a$ in $\cO_{K_0}$, so $M/\ff_0M \cong \mathfrak a/\ff_0\mathfrak a \cong R$. So there is an element $\alpha \in M$ such that $M/\ff_0M = R\alpha$.
\end{proof}

\subsection*{Proof of Theorem~\ref{thm:classificationOrdersMaxRM}}
For~\ref{thmClassificationMaxRM1}, let $\ff_0$ be an ideal in $\cO_{K_0}$, and write $\ff = \ff_0\cO_K$. Let $\mathfrak c$ be the conductor of $\cO_{K_0} + \ff$. From Lemma~\ref{maxRM_rosatiStable}, $\mathfrak c$ is of the form $\mathfrak c_0 \cO_K$ where $\mathfrak c_0 = \cO_{K_0} \cap \mathfrak c$. Clearly $\ff \subset \mathfrak c$, so $\mathfrak c_0 \mid \ff_0$ and we can write $\ff_0 = \mathfrak c_0 \mathfrak g_0$. Let $\pi : \cO_K \rightarrow~\cO_K/\ff$ be the canonical projection.
Since $\mathfrak c \subset \cO_{K_0} + \ff$, we have $\pi(\mathfrak c) \subset \pi(\cO_{K_0})$. From Lemma~\ref{bijfrak},
$$\left|\pi(\mathfrak c) \right| = \left| \pi(\cO_{K_0}) \cap \pi(\mathfrak c) \right| = |\mathfrak  c_0/\ff_0| = N(\mathfrak g_0).$$
On the other hand,
$\left| \pi(\mathfrak c) \right| = |\mathfrak  c/\ff| = N(\mathfrak g_0\cO_K) = N(\mathfrak g_0)^2,$
so $N(\mathfrak g_0)=1$, hence $\mathfrak c = \ff$.
To prove~\ref{thmClassificationMaxRM2}, let $\cO$ be an order in $K$ with maximal real multiplication and conductor $\ff$. From Lemma~\ref{maxRM_rosatiStable}, $\cO$ is $\dagger$-stable and $\ff = \ff_0\cO_K$, where $\ff_0 = \ff~\cap~\cO_{K_0}$.
We claim that if $x \in \cO$ then $x \in \cO_{K_0} + \ff$. Let $R = \cO_{K_0}/\ff_0$. By Lemma~\ref{lemmaModuleQuotient}, $\cO_K/\ff = R \oplus R \alpha$. The quotient $\cO/\ff$ is an $R$-submodule of $\cO_K/\ff$.

There are two elements $y,z \in R$ such that $x + \ff = y + z\alpha$. Then, $z \alpha \in \cO/\ff$, and 
we obtain that $(zR)\alpha \subset \cO/\ff$. There exists an ideal $\mathfrak g_0$ dividing $\ff_0$ such that $zR = \mathfrak g_0 / \ff_0$. Therefore $(\mathfrak g_0 / \ff_0)\alpha \subset \cO/\ff$. Then,
$$ \mathfrak g / \ff \subset R +  (\mathfrak g_0/ \ff_0) \alpha \subset \cO/\ff,$$
where $\mathfrak g = \mathfrak g_0 \cO_K$, which implies that $\mathfrak g \subset \cO$. But $\mathfrak g$ divides $\ff$, and $\ff$ is the largest $\cO_K$-ideal in $\cO$, so $\mathfrak g = \ff$. Hence $z \in \ff$, and $x \in \cO_{K_0} + \ff$.\qed

\section{From abelian surfaces to lattices, and vice-versa}\label{sec:correspondence}

\subsection{Tate modules and isogenies} 
Consider again the setting introduced in Subsection~\ref{subsec:setting}, with $\mathscr A$ an abelian variety over the finite field $k$ in the fixed isogeny class --- ordinary, absolutely simple, and of dimension $g$.
Write $T = T_\ell \mathscr A$ for the $\ell$-adic Tate module of $\mathscr A$, and $V$ for $T \otimes_{\Z_\ell} \Q_\ell$. Then $V$ is a $2g$-dimensional $\mathbb{Q}_\ell$-vector space with an action of the algebra $K_\ell$, over which it has rank one, and $T$ is similarly of rank one over the ring $\lO(\mathscr A) = \cO(\mathscr A) \otimes_{\Z} \Z_\ell$. Write $\pi$ for the Frobenius endomorphism of $\mathscr A$, viewed as an element of $\mathcal{O}(\mathscr A)$.

The elements of $T$ are the sequences $(Q_n)_{n \geq 0}$ with $Q_n \in \mathscr A[\ell^n]$, $\ell Q_n = Q_{n-1}$ for all $n \geq 1$. 
An element of $V$ identifies with a sequence $(P_n)_{n \geq 0}$ with $P_n \in \mathscr A[\ell^\infty]$ and $\ell P_n = P_{n-1}$ for $n \geq 1$ as follows:
$$
(Q_n)_{n \geq 0} \otimes \ell^{-m} \longmapsto (Q_{n+m})_{n \geq 0},
$$
and under this identification, $T$ is the subgroup of $V$ where $P_0 = 0 \in \mathscr A[\ell^\infty]$. The projection to the zeroth coordinate then yields a canonical identification
\begin{equation}\label{eq:identification}
V/T \stackrel{\sim}{\longrightarrow} \mathscr A[\ell^\infty](\overline{k}),
\end{equation}
under which the action of $\pi$ on the left-hand side corresponds to the action of the arithmetic Frobenius element in $\Gal(\overline{k}/k)$ on the right-hand side.

We are now ready to state the main correspondence between lattices in $V$ containing the Tate module $T$ and $\ell$-power isogenies from $\mathscr A$.

\begin{proposition}\label{prop:correspondence}
There is a one-to-one correspondence
$$
\left\{\text{Lattices in } V \text{ containing } T \right\} \cong
\left\{\mbox{finite subgroups of $\mathscr A[\ell^\infty]$}\right\}, 
$$ 
where a lattice $\Gamma$ is sent to the subgroup $\Gamma/T$, through the identification \eqref{eq:identification}. Under this correspondence,
\begin{enumerate}[label=(\roman*)]
\item A lattice is stable under $\pi^n$ if and only if the corresponding subgroup is defined over the degree $n$ extension $\F_{q^n}$ of $k$.
\item If a subgroup $\kappa \subset \mathscr A[\ell^\infty]$ corresponds to a lattice $\Gamma$, then the order of $K_\ell$ of elements stabilizing $\Gamma$ is $\lO(\mathscr A/\kappa)$.
\end{enumerate}
\end{proposition}

\begin{proof}
A lattice $\Gamma$ in $V$ is sent to the subgroup $\kappa$ of $\mathscr A[\ell^\infty]$ corresponding to $\Gamma$ under~\eqref{eq:identification}. Conversely, give a subgroup $\kappa \subset \mathscr A[\ell^\infty]$, let $\Gamma$ be the set of sequences in $V$ whose zeroth coordinate is in $\kappa$. It follows that the subgroup of $\mathscr A[\ell^\infty]$ corresponding to this lattice under \eqref{eq:identification} is $\kappa$, so that this process is indeed bijective. 

The claim about fields of definition follows from the previously-discussed Frobenius equivariance of \eqref{eq:identification}. The claim about endomorphism rings is Tate's isogeny theorem applied to $\Hom(\mathscr A/\kappa, \mathscr A/\kappa)$.
\end{proof}

\begin{remark}\label{rem:kernelVsIso}
Observe that given a subgroup $\kappa \subset \mathscr A[\ell^\infty]$, any two isogenies of kernel $\kappa$ differ only by an isomorphism between the targets.
Therefore if $\varphi : \mathscr A \rightarrow \mathscr B$ is any isogeny of kernel $\kappa$, then $\lO(\mathscr A/\kappa) = \lO(\mathscr B)$.
\end{remark}

\begin{remark}\label{rem:fieldOfDefIsogenies}
Recall that all varieties and morphisms are considered over $\overline k$. We are however also interested in the structures arising when restricting to varieties and morphisms defined over $k$, in the sense of Subsection~\ref{subsec:setting}. 
To this end, the most important fact (which is special to the case of simple, ordinary abelian varieties) is that if a variety $\mathscr B$ is $k$-isogenous to $\mathscr A$, then any isogeny $\mathscr A \ra \mathscr B$ is defined over $k$ 
(this is an easy consequence of~\cite[Thm.7.2.]{Waterhouse1969}: if $\mathscr B$ is defined over $k$, and $\varphi, \psi : \mathscr A \ra \mathscr B$ are two isogenies  then $\varphi \circ \psi^{-1}$ is an element of $\End(\mathscr B)\otimes_\Z\Q$, hence defined over $k$, so $\varphi$ is defined over $k$ if and only if $\psi$ is).
Similarly to Remark~\ref{rem:kernelVsIso}, if $\kappa$ is defined over $k$, any two $k$-isogenies of kernel $\kappa$ differ by a $k$-isomorphism between the targets. From Proposition~\ref{prop:correspondence}(ii), if $\pi \in \lO(\mathscr A/\kappa)$, then $\kappa$ is defined over $k$, and is thereby the kernel of a $k$-isogeny\footnote{Note that in general, if $\mathscr B$ is $\overline k$-isogenous to $\mathscr A$ and $\pi \in \cO(\mathscr B)$, then $\pi$ does not necessarily correspond to the $k$-Frobenius of $\mathscr B$ unless $\mathscr B$ is actually $k$-isogenous to $\mathscr A$.}. We obtain a correspondence between subgroups $\kappa$ defined over $k$ and $\Gamma$ lattices stabilized by $\pi$.
\end{remark}

The following proposition justifies the strategy of working locally at $\ell$, as it guarantees that $\ell$-power isogenies do not affect endomorphism rings at primes $\ell' \neq \ell$.
\begin{proposition}\label{prop:EllDoesNotChangeP}
Let $\varphi: \mathscr A \to \mathscr B$ be an isogeny of abelian varieties of $\ell$-power degree. Then for any prime $\ell' \neq \ell$ of $\mathscr A$, one has $\cO(\mathscr A) \otimes_\Z \Z_{\ell'} = \cO(\mathscr B) \otimes_\Z \Z_{\ell'}$.
\end{proposition}

\begin{proof}
Let $\mathcal{C}_{\ell'}$ be the category whose objects are abelian varieties over $\overline k$ and whose morphisms are $\Hom_{\mathcal{C}_{\ell'}}(\mathscr A_1, \mathscr A_2) = \Hom(\mathscr A_1, \mathscr A_2) \otimes_\Z \Z_{\ell'}$. There exists an isogeny $\hat\varphi: \mathscr B \to \mathscr A$ such that $\hat \varphi \circ \varphi = [\ell^n]$, so $\varphi$ induces an isomorphism in $\mathcal{C}_{\ell'}$; it follows that the endomorphism rings of $\mathscr A$ and $\mathscr B$ in this category are identified.
\end{proof}

\subsection{Polarizations and symplectic structures}
Fix a polarization $\xi$ of $\mathscr A$. It induces a polarization isogeny $\lambda:\mathscr A \to \mathscr A^\vee$, which in turn gives a map $T \to T_\ell(\mathscr A^\vee)$. Therefore the Weil pairing equips $T$ with a natural $\Z_\ell$-linear pairing $\langle-,-\rangle$, which extends to a pairing on $V$. We gather standard facts about this pairing in the following lemma.

\begin{lemma}
One has:
 \begin{enumerate}[label=(\roman*)]
\item The pairing $\langle-,-\rangle$ is symplectic.
\item For any $\alpha \in K$, one has
$$
\langle \alpha x, y \rangle = \langle x, \alpha^\dagger y \rangle,
$$
where $\dagger$ denotes the complex conjugation.
\item For $\Gamma$ a lattice in $V$, write
$$
\Gamma^* = \{ \alpha \in V \mid \langle \alpha, \Gamma \rangle \subset \Z_\ell \}
$$
for the dual lattice of $\Gamma$. Then $T \subset T^*$, and the quotient is isomorphic to $(\ker \lambda)[\ell^\infty]$. In particular, $T$ is self-dual if and only if the degree of $\lambda$ is coprime to $\ell$.
\end{enumerate}
\end{lemma}

\begin{proof}
The first two claims are standard --- see \cite[Lemma 16.2e, and \textsection 167]{MilneAV}. For the third,
note that $T^*$ identifies with $\lambda_*^{-1}(T_\ell \mathscr A^\vee)$, and $\lambda_*$ induces an isomorphism
$$\lambda_*^{-1}(T_\ell \mathscr A^\vee)/{T} \stackrel{\sim}{\longrightarrow} \ker (\lambda)[\ell^\infty].$$
\end{proof}

\section{Graphs of $\mathfrak l$-isogenies}

In this section we study $\mathfrak l$-isogenies through the lens of lattices in an $\ell$-adic vector space, endowed with an action of the algebra $K_\ell$.

\subsection{Lattices with locally maximal real multiplication}
Throughout this subsection, $V$ is a $\Q_\ell$-vector space of dimension $2g$, $\ell$ is a prime number, $K$ is a quartic CM-field, with $K_{0}$ its maximal real subfield. The algebra $K_\ell$ is a $\Q_\ell$-algebra of dimension~$2g$.  Suppose that it acts ($\Q_\ell$-linearly) on~$V$.
Define the \emph{order} of a full-rank $\Z_\ell$-lattice $\Lambda \subset V$ as
$$\lO(\Lambda) = \{x \in K_\ell \mid x\Lambda \subset \Lambda \}.$$
For any order $\lO$ in $K_\ell$, say that $\Lambda$ is an $\lO$-lattice if $\lO(\Lambda) = \lO$.
Let $\lO = \lO(\Lambda)$ be the order of $\Lambda$, and suppose that it has maximal real multiplication, i.e., that $\lO$ contains the maximal order $\lO_0$ of $K_{0,\ell} = K_{0} \otimes_\Q \Q_\ell$.
We now need some commutative algebra:

\begin{lemma}\label{lemma:quadraticImpliesGorenstein}
Let $A$ be a Dedekind domain with field of fractions $F$, and let $L$ be a quadratic extension of $F$. If $\cO$ is any $A$-subalgebra of the integral closure of $A$ in $L$, with $\cO \otimes K = L$, then $\cO$ is Gorenstein. 
\end{lemma}

\begin{proof}
The hypotheses and result are local on $\text{Spec} A$, so we may take $A$ a principal ideal domain. Then $\cO$ is a free $A$-module, which must be $2$-dimensional. The element $1 \in \cO$ is not an $A$-multiple of any element of $\cO$, so there is a basis $\{1, \alpha\}$ for $\cO$ as an $A$-module; clearly $\cO = A[\alpha]$ as $A$-algebras. The result then follows from \cite[Ex.2.8]{BuchmannLenstra}.
\end{proof}

By Lemma \ref{lemma:quadraticImpliesGorenstein}, the order $\lO$, which has maximal real multiplication, is a Gorenstein ring and $\Lambda$ is a free $\lO$-module of rank 1.
Recall the notations $\lO_K = \cO_K\otimes_\Z \Z_\ell$ and $\lO_0 = \cO_{K_0}\otimes_\Z \Z_\ell$ from Section~\ref{subsec:setting}.
For any ideal $\mathfrak f$ in $\lO_0$, let $\lO_\mathfrak f = \lO_0 + \mathfrak f\lO_{K}$. From Theorem~\ref{thm:classificationOrdersMaxRM}, all the orders containing $\lO_0$ are of this form.

\begin{definition}[$\mathfrak l$-neighbors]
Let $\Lambda$ be a lattice with maximal real multiplication, and let $\mathfrak l$ be a prime ideal in $\lO_{0}$. The set $\mathscr L_{\mathfrak l}(\Lambda)$ of \emph{$\mathfrak l$-neighbors} of $\Lambda$ consists
of all the lattices $\Gamma$ such that $\mathfrak l \Lambda \subset \Gamma \subset \Lambda$, and $\Gamma/\mathfrak l\Lambda \cong \lO_{0}/\mathfrak l$, i.e., $\Gamma/\mathfrak l\Lambda \in \mathbb P^1(\Lambda/\mathfrak l\Lambda)$.
\end{definition}

\begin{remark}\label{rem:correspFrakLAndFrakL}
Consider the lattice $T = T_\ell\mathscr A$.
Then, $\mathfrak l$-isogenies $\mathscr A \to \mathscr B$ (see Definition~\ref{def:frakLIso}) correspond under Proposition \ref{prop:correspondence} to lattices $\Gamma$ with $T \subset \Gamma \subset \mathfrak l^{-1} T$ and $\Gamma/T$ is an $\lO_{0}/\mathfrak l$-subspace of dimension one of $(\mathfrak l^{-1} T) / T$.
\end{remark}

The following lemma is key to understanding $\mathfrak l$-neighbors. It arises from the technique employed by Cornut and Vatsal~\cite[\textsection 6]{Cornut04} to study the action of a certain Hecke algebra on quadratic CM-lattices.

\begin{lemma}\label{lemma:fixedPointsMaxRM}
Let $K$ be a CM-field, and $K_0$ its maximal real subfield. Let $\mathfrak l$ be a prime ideal in $\lO_{0}$, and $\mathbb F = \lO_{0}/\mathfrak l$. Let $\mathfrak f$ be an ideal in $\lO_{0}$ and $\lO_\mathfrak f = \lO_{0} + \mathfrak f \lO_{K}$.
The action of $\lO_{\mathfrak f}^\times$ on the set of $\F$-lines $\mathbb P^1(\lO_{\mathfrak f}/\mathfrak l \lO_{\mathfrak f})$ factors through $\lO_{\mathfrak f}^\times/ \lO_{\mathfrak l\mathfrak f}^\times$. Let $\mathfrak L$ be a prime in $\lO_{\mathfrak f}$ above $\mathfrak l$. The fixed points are
\[
\mathbb P^1(\lO_{\mathfrak f}/\mathfrak l \lO_{\mathfrak f})^{\lO_{\mathfrak f}^\times} = 
\left\{ \begin{array}{ll}
        \emptyset & \mbox{if $\mathfrak l \nmid \mathfrak f$ and $\mathfrak l\lO_\mathfrak f = \mathfrak L$,}\\
        \{\mathfrak L/\mathfrak l \lO_{\mathfrak f}, \mathfrak L^\dagger/\mathfrak l \lO_{\mathfrak f}\} & \mbox{if $\mathfrak l \nmid \mathfrak f$ and $\mathfrak l\lO_\mathfrak f = \mathfrak L \mathfrak L^\dagger$,}\\
        \{(\mathfrak l \lO_{\mathfrak l^{-1}\mathfrak f})/\mathfrak l \lO_{\mathfrak f}\} & \mbox{if $\mathfrak l \mid \mathfrak f$.}\end{array} \right.
\]
The remaining points are permuted simply transitively by $\lO_{\mathfrak f}^\times/ \lO_{\mathfrak l\mathfrak f}^\times$.
\end{lemma}

\begin{proof}
The ring $\lO_{\mathfrak l\mathfrak f}^\times$ acts trivially on $\mathbb P^1(\lO_{\mathfrak f}/\mathfrak l \lO_{\mathfrak f})$, which proves the first statement. 
Observe that the projection $\lO_{\mathfrak f} \rightarrow \lO_{\mathfrak f}/\lO_{\mathfrak l\mathfrak f}$ induces a canonical isomorphism between $\lO_{\mathfrak f}^\times/ \lO_{\mathfrak l\mathfrak f}^\times$ and $(\lO_{\mathfrak f}/\mathfrak l \lO_{\mathfrak f})^\times/\F^{\times}$.
Suppose that $\mathfrak l$ divides $\mathfrak f$. Then, there exists an element $\epsilon \in \lO_{\mathfrak f}/\mathfrak l \lO_{\mathfrak f}$ such that $\lO_{\mathfrak f}/\mathfrak l \lO_{\mathfrak f} = \F[\epsilon]$ and $\epsilon^2 = 0$. But the only $\F$-line in $\F[\epsilon]$ fixed by the action of $\F[\epsilon]^\times$ is $\epsilon \F = (\mathfrak l \lO_{\mathfrak l^{-1}\mathfrak f})/\mathfrak l \lO_{\mathfrak f}$, and this action is transitive on the $\ell$ other lines. Therefore the action of $\F[\epsilon]^\times/\F^\times = (\lO_{\mathfrak f}/\mathfrak l \lO_{\mathfrak f})^\times/\F^{\times}$ on these $\ell$ lines is simply transitive. 

Now, suppose that $\mathfrak l$ does not divide $\mathfrak f$. If $\mathfrak l$ is inert in $\lO_\mathfrak f$, then $\lO_{\mathfrak f}/\mathfrak l \lO_{\mathfrak f} = \mathbb K$ is a quadratic field extension of $\F$, and $\mathbb K^\times/\F^\times$ acts simply transitively on the $\F$-lines $\mathbb P^1(\mathbb K)$. To statement follows from the isomorphism between $\mathbb K^\times/\F^\times$ and $\lO_{\mathfrak f}^\times/ \lO_{\mathfrak l\mathfrak f}^\times$.
The cases where $\mathfrak l$ splits or ramifies in $K$ are treated similarly, with $\lO_{\mathfrak f}/\mathfrak l \lO_{\mathfrak f} \cong \F^2$ in the first case, and $\lO_{\mathfrak f}/\mathfrak l \lO_{\mathfrak f} \cong \F[X]/(X^2)$ in the second case.
\end{proof}

\begin{proposition}[Structure of $\mathscr L_\mathfrak l(\Lambda)$]\label{prop:genericNeighbors}
Suppose $\Lambda$ is an $\lO_{\mathfrak f}$-lattice, for some $\lO_{0}$-ideal $\mathfrak f$, and let $\mathfrak l$ be a prime ideal in $\lO_{0}$. The lattice $\Lambda$ has $N(\mathfrak l) + 1$ $\mathfrak l$-neighbors. The $\mathfrak l$-neighbors that have order $\lO_{\mathfrak l\mathfrak f}$ are permuted simply transitively by $(\lO_{\mathfrak f} /  \lO_{\mathfrak l\mathfrak f})^\times$. The other $\mathfrak l$-neighbors have order $\lO_{\mathfrak l^{-1}\mathfrak f}$ if $\mathfrak l $ divides $ \mathfrak f$, or $\lO_{K}$ otherwise.

More explicitly, if $\mathfrak l $ divides $ \mathfrak f$, there is one $\mathfrak l$-neighbor of order $\lO_{\mathfrak l^{-1}\mathfrak f}$, namely $\mathfrak l \lO_{\mathfrak l^{-1}\mathfrak f} \Lambda$, and $N(\mathfrak l)$ $\mathfrak l$-neighbors of order $\lO_{\mathfrak l\mathfrak f}$. If $\mathfrak l $ does not divide $ \mathfrak f$, we have:
\begin{enumerate}[label=(\roman*)]
\item If $\mathfrak l$ is inert in $K$, all $N(\mathfrak l) + 1$ lattices of $\mathscr L_\mathfrak l(\Lambda)$ have order $\lO_{\mathfrak l}$,
\item If $\mathfrak l$ splits in $K$ into prime ideals $\mathfrak L_1$ and $\mathfrak L_2$, $\mathscr L_\mathfrak l(\Lambda)$ consists of two lattices of order $\lO_{K}$, namely $\mathfrak L_1\Lambda$ and $\mathfrak L_2\Lambda$, and $N(\mathfrak l)-1$ lattices of order $\lO_{\mathfrak l}$,
\item If $\mathfrak l$ ramifies in $K$ as $\mathfrak L^2$, $\mathscr L_\mathfrak l(\Lambda)$ consists of one lattice of order $\lO_{K}$, namely $\mathfrak L\Lambda$, and $N(\mathfrak l)$ lattices of order $\lO_{\mathfrak l}$.
\end{enumerate}
\end{proposition}

\begin{proof}
This is a direct consequence of Lemma~\ref{lemma:fixedPointsMaxRM}, together with the fact that $\Lambda$ is a free $\lO_{\mathfrak f}$-module of rank 1.
\end{proof}

\subsection{Graphs of $\mathfrak l$-isogenies}\label{subsec:frakLIsogenyGraphs}

Fix again a principally polarizable absolutely simple ordinary abelian variety $\mathscr A$ of dimension $g$ over $k$, with endomorphism algebra~$K$. 
Suppose that $\mathscr A$ has locally maximal real multiplication at $\ell$ (i.e., $\lO_0 \subset \lO(\mathscr A)$). The $\mathfrak l$-neighbors correspond in the world of varieties to $\mathfrak l$-isogenies (see Remark~\ref{rem:correspFrakLAndFrakL}).

\begin{definition}
Suppose $\mathscr A$ has local order $\lO_{\mathfrak f}$, for some $\lO_{0}$-ideal $\mathfrak f$ and let $\mathfrak l$ be a prime ideal in $\lO_{0}$. An $\mathfrak l$-isogeny $\varphi : \mathscr A \rightarrow \mathscr B$ is \emph{$\mathfrak l$-ascending} if $\lO(\mathscr B) = \lO_{\mathfrak l^{-1}\mathfrak f}$, it is \emph{$\mathfrak l$-descending} if $\lO(\mathscr B) = \lO_{\mathfrak l\mathfrak f}$, and it is \emph{$\mathfrak l$-horizontal} if $\lO(\mathscr B) = \lO_{\mathfrak f}$.
\end{definition}

\begin{proposition}\label{prop:frakLStructure}
Suppose $\mathscr A$ has local order $\lO_{\mathfrak f}$ for some $\lO_{0}$-ideal $\mathfrak f$  and let $\mathfrak l$ be a prime ideal in $\lO_{0}$. There are $N(\mathfrak l) + 1$ kernels of $\mathfrak l$-isogenies from $\mathscr A$. The kernels of $\mathfrak l$-descending $\mathfrak l$-isogenies are permuted simply transitively by the action of $(\lO_{\mathfrak f} /  \lO_{\mathfrak l\mathfrak f})^\times$. The other $\mathfrak l$-isogenies are $\mathfrak l$-ascending if $\mathfrak l $ divides $ \mathfrak f$, and $\mathfrak l$-horizontal otherwise.

More explicitely, if $\mathfrak l $ divides $ \mathfrak f$, there is a unique $\mathfrak l$-ascending $\mathfrak l$-kernel from $\mathscr A$, 
and $N(\mathfrak l)$ $\mathfrak l$-descending $\mathfrak l$-kernels. If $\mathfrak l$ does not divide $ \mathfrak f$, we have:
\begin{enumerate}[label=(\roman*)]
\item If $\mathfrak l$ is inert in $K$, all $N(\mathfrak l) + 1$ $\mathfrak l$-kernels are $\mathfrak l$-descending;
\item If $\mathfrak l$ splits in $K$ into two prime ideals $\mathfrak L_1$ and $\mathfrak L_2$, there are two $\mathfrak l$-horizontal $\mathfrak l$-kernels, namely $\mathscr A[\mathfrak L_1]$ and $\mathscr A[\mathfrak L_2]$, and $N(\mathfrak l)-1$ $\mathfrak l$-descending ones;
\item If $\mathfrak l$ ramifies in $K$ as $\mathfrak L^2$, there is one $\mathfrak l$-horizontal $\mathfrak l$-kernel, namely $\mathscr A[\mathfrak L]$, and $N(\mathfrak l)$ $\mathfrak l$-descending ones.
\end{enumerate}
\end{proposition}

\begin{proof}
This proposition follows from Proposition~\ref{prop:genericNeighbors} together with Remark~\ref{rem:correspFrakLAndFrakL}.
\end{proof}

\begin{definition}[$\mathfrak l$-predecessor]
When it exists, let $\kappa$ be the unique $\mathfrak l$-ascending kernel of Proposition~\ref{prop:frakLStructure}. We call 
$\pr_\mathfrak l(\mathscr A) = \mathscr A / \kappa$ the \emph{$\mathfrak l$-predecessor} of $\mathscr A$, and denote by
$\mathrm{up}_{\mathscr A}^{\mathfrak l} : \mathscr A \ra \pr_\mathfrak l(\mathscr A)$ the canonical projection.
\end{definition}

The following notion of volcano was introduced in~\cite{fouquet-morain} to describe the structure of graphs of $\ell$-isogenies between elliptic curves.

\begin{definition}[volcano]
Let $n$ be a positive integer. An (infinite) \emph{$n$-volcano} $\mathscr V$ is an $(n+1)$-regular, connected, undirected graph whose vertices are partitioned into \emph{levels} $\{\mathscr V_i\}_{i \in \Z_{\geq 0}}$ such that:
\begin{enumerate}[label=(\roman*)]
\item The subgraph $\mathscr V_0$, the \emph{surface}, is a finite regular graph of degree at most 2,
\item For each $i > 0$, each vertex in $\mathscr V_i$ has exactly one neighbor in $\mathscr V_{i-1}$, and these are exactly the edges of the graph that are not on the surface.
\end{enumerate}
For any positive integer $h$, the corresponding (finite) volcano of height $h$ is the restriction of $\mathscr V$ to its first $h$ levels.
\end{definition}

Let $\mathfrak l$ be a prime of $K_0$ above $\ell$.
Consider the $\mathfrak l$-isogeny graph $\mathscr W_\mathfrak l$ as defined in Section~\ref{subsubsec:theorem1}.
Note that it is a directed multigraph; we say that such a graph is \emph{undirected} if for any vertices $u$ and $v$, the multiplicity of the edge from $u$ to $v$ is the same as the multiplicity from $v$ to $u$.
The remainder of this section is a proof of Theorem~\ref{thm:lisogenyvolcanoes}, which provides a complete description of the structure of the leveled $\mathfrak l$-isogeny graph $(\mathscr W_\mathfrak l, v_\mathfrak l)$, closely related to volcanoes.

\begin{lemma}\label{lemma:distinctFrakLKernels}
Suppose that $\cO(\mathscr B)\subset \cO(\mathscr A)$. If there exists an $\mathfrak l$-isogeny $\varphi: \mathscr A \to \mathscr B$, then there are at least $[\cO(\mathscr A)^\times : {\cO}(\mathscr B)^\times]$ pairwise distinct kernels of $\mathfrak l$-isogenies from $\mathscr A$ to $\mathscr B$.
\end{lemma}

\begin{proof}
The elements $\alpha \in \cO(\mathscr A)$ act on the subgroups of $\mathscr A$ via the isomorphism $\cO(\mathscr A) \cong \End(\mathscr A)$, and we denote this action $\kappa \mapsto \kappa^\alpha$. Let $\kappa = \ker \varphi$. If $u \in  \cO(\mathscr A)^\times$ is a unit, then $\kappa^u$ is also the kernel of an $\mathfrak l$-isogeny. Furthermore, $u$ canonically induces an isomorphism $\mathscr A/\kappa \rightarrow \mathscr A/\kappa^u$, so $\kappa^u$ is the kernel of a $\mathfrak l$-isogeny with target $\mathscr B$.

It only remains to prove that the orbit of $\kappa$ for the action of $\cO(\mathscr A)^\times$ contains at least $[\cO(\mathscr A)^\times : {\cO}(\mathscr B)^\times]$ distinct kernels. It suffices to show that if $\kappa^u = \kappa$, then $u \in \cO(\mathscr B)^\times$. Let $u \in \cO(\mathscr A)^\times$ such that $\kappa^u = \kappa$.
Recall that for any variety $\mathscr C$ in our isogeny class, we have fixed an isomorphism $\imath_\mathscr C : \End(\mathscr C) \ra \cO(\mathscr C)$, and that these isomorphisms are all compatible in the sense that for any isogeny $\psi : \mathscr C \ra \mathscr D$, and $\gamma \in \End(\mathscr C)$, we have $\imath_\mathscr C(\gamma) = \imath_\mathscr D(\psi \circ \gamma \circ \hat\psi)/\deg \psi$.
Let $u_\mathscr A \in \End(\mathscr A)$ be the endomorphism of $\mathscr A$ corresponding to~$u$. It induces an isomorphism $\tilde u_\mathscr A : \mathscr A/\kappa \ra \mathscr A/\kappa^u$, which is actually an automorphism of $\mathscr A/\kappa$ since $\kappa^u = \kappa$. Let $\varphi : \mathscr A \ra \mathscr A/\kappa$ be the natural projection. We obtain the following commutative diagram:
\begin{equation*}
\xymatrix{
\mathscr A/\kappa \ar[r]^{\tilde u_\mathscr A} & \mathscr A/\kappa \ar[rd]^{\hat \varphi} & \\
\mathscr A \ar[u]^{\varphi}  \ar[r]^{u_\mathscr A} & \mathscr A \ar[u]^{\varphi} \ar[r]_{[\deg \varphi]} & \mathscr A.} 
\label{eq:ppav}
\end{equation*}
Finally, we obtain
\[u = \imath_\mathscr A([\deg \varphi] \circ u_\mathscr A)/\deg \varphi = \imath_\mathscr A( \hat\varphi \circ \tilde u_\mathscr A \circ \varphi)/\deg \varphi = \imath_\mathscr B(\tilde u_\mathscr A)\in \cO(\mathscr B).\]
\end{proof}
\begin{lemma}\label{lemma:classNumberRelation}
Let $K$ be a CM-field and $K_0$ its maximal real subfield.
Let $\cO$ be an order in $K$ of conductor $\mathfrak f$ such that $\lO_{0} \subset \cO \otimes_\Z \Z_{\ell}$. Let $\cO'$ be the order such that $\cO' \otimes_\Z \Z_{\ell'} = \cO \otimes_\Z \Z_{\ell'}$ for all prime $\ell' \neq \ell$, and $\cO' \otimes_\Z \Z_{\ell} = \lO_{0} + \mathfrak l\mathfrak f\lO_{K}$. Then, \[|\Pic(\cO')| = \frac{\left[(\cO \otimes_\Z \Z_{\ell})^\times : ({\cO'} \otimes_\Z \Z_\ell)^\times\right]}{[\cO^\times : {\cO'}^\times]}|\Pic(\cO)|.\]
\end{lemma}

\begin{proof}
First, for any order $\cO$ in $K$ of conductor $\mathfrak f$ we have the classical formula (see~\cite[Th.12.12 and Prop.12.11]{Neukirch99})
\begin{align*}
|\Pic(\cO)| &= \frac{h_K}{[\cO_K^\times : \cO^\times]}\frac{|(\cO_K/\mathfrak f)^\times|}{|(\cO/\mathfrak f)^\times|}\\
& = \frac{h_K}{[\cO_K^\times : \cO^\times]} \prod_{\ell' \text{ prime}} [(\cO_{K} \otimes_\Z \Z_{\ell'})^\times : (\cO \otimes_\Z \Z_{\ell'})^\times].
\end{align*}
Now, consider $\cO$ and $\cO'$ as in the statement of the lemma. We obtain
\begin{align*}
\frac{|\Pic(\cO')|}{|\Pic(\cO)|} & = \frac{[\cO_K^\times : \cO^\times]}{[\cO_K^\times : {\cO'}^\times]} [(\cO \otimes_\Z \Z_{\ell})^\times : (\cO' \otimes_\Z \Z_{\ell})^\times] \\
& = \frac{\left[(\cO \otimes_\Z \Z_{\ell})^\times : (\cO' \otimes_\Z \Z_\ell)^\times\right]}{[\cO^\times : {\cO'}^\times]}.
\end{align*}
\end{proof}

\begin{remark}
If one supposes that $\cO_K^\times = \cO_{K_0}^\times$, then $[\cO^\times : {\cO'}^\times]$ is always $1$ in the above lemma. Indeed, one has  $\cO^\times \subset  \cO_{K_0}^\times \subset \lO_{0}^\times \subset (\cO' \otimes_\Z \Z_\ell)^\times,$ and
therefore, since $\cO$ and $\cO'$ coincide at every other prime, we obtain
$\cO^\times \subset  {\cO'}^\times,$ hence $\cO^\times =  {\cO'}^\times$.
\end{remark}
\begin{remark}\label{rem:Streng}
For $g = 2$, the field $K$ is a primitive quartic CM-field. Then, the condition $\cO_K^\times = \cO_{K_0}^\times$ is simply equivalent to $K \neq \Q(\zeta_5)$ by \cite[Lem.3.3]{Streng10}. So in dimension 2, if $K \neq \Q(\zeta_5)$, one always has $[\cO^\times : {\cO'}^\times] = 1$ in the above lemma.
\end{remark}

\subsection*{Proof of Theorem~\ref{thm:lisogenyvolcanoes}}
Let $\mathscr V$ be any of connected component of $\mathscr W_\mathfrak l$.
First, it follows from Proposition~\ref{prop:EllDoesNotChangeP} that locally at any prime other than $\ell$, the endomorphism rings occurring in $\mathscr V$ all coincide. Also, locally at $\ell$, Proposition~\ref{prop:frakLStructure} implies that an $\mathfrak l$-isogeny can only change the valuation at $\mathfrak l$ of the conductor. Therefore within $\mathscr V$, the endomorphism ring of a variety $\mathscr A$ is uniquely determined by its level $v_{\mathfrak l}(\mathscr A)$.  Let $\cO_i$ be the endomorphism of any (and therefore every) variety $\mathscr A$ in $\mathscr V$ at level $v_{\mathfrak l}(\mathscr A) = i$. Write $\mathscr V_i$ for the corresponding subset of $\mathscr V$.
Proposition~\ref{prop:frakLStructure} implies that, except at the surface, all the edges connect consecutive levels of the graph, and each vertex at level $i$ has exactly one edge to the level $i-1$.

The structure of the connected components of the level $\mathscr V_0$ is already a consequence of the well-known free CM-action of $\Pic(\cO_0)$ on ordinary abelian varieties with endomorphism ring $\cO_0$.
Note that if $\varphi: \mathscr A \rightarrow \mathscr B$ is a descending $\mathfrak l$-isogeny within $\mathscr V$, then the unique ascending $\mathfrak l$-isogeny from $\mathscr B$ is $\mathrm{up}_{\mathscr B}^\mathfrak l : \mathscr B \rightarrow \pr_\mathfrak l(\mathscr B)$, and we
have $\pr_\mathfrak l(\mathscr B) \cong \mathscr A/\mathscr A[\mathfrak l]$; also, we have
$\pr_\mathfrak l(\mathscr B/\mathscr B[\mathfrak l]) \cong \pr_\mathfrak l(\mathscr B)/\pr_\mathfrak l(\mathscr B)[\mathfrak l]$.
These facts easily follow from the lattice point of view (see Proposition~\ref{prop:genericNeighbors}, and observe that if
$\Gamma \in \mathscr L_\mathfrak l (\Lambda)$, then $\mathfrak l\Gamma \in \mathscr L_\mathfrak l (\mathfrak l\Lambda)$). 
We can deduce in particular that $\mathscr V_0$ is connected: a path from $\mathscr A \in \mathscr V_0$ to another vertex of $\mathscr V_0$ containing only vertical isogenies can only end at a vertex $\mathscr A/\mathscr A[\mathfrak l^i]$, which can also be reached within $\mathscr V_0$.

We now need to look at a bigger graph. For each $i \geq 0$, let $\mathscr U_{i}$ be the orbit of the level $\mathscr V_i$ for the CM-action of $\Pic(\cO_i)$. The action is transitive on $\mathscr U_0$ since the connected graph $\mathscr V_0$ is in a single orbit of the action of $\Pic(\cO_0)$. Let us show by
induction that each $\mathscr U_{i+1}$ consists of a single orbit, and that each vertex of $\mathscr U_{i+1}$ is reachable by an edge from $\mathscr U_{i}$. First, $\mathscr U_{i+1}$ is non-empty because,
by induction, $\mathscr U_{i}$ is non-empty, and each vertex in $\mathscr U_{i}$ has neighbors in $\mathscr U_{i+1}$. Choose
any isogeny $\varphi : {\mathscr A}'\rightarrow \mathscr A$ from $\mathscr U_{i}$ to $\mathscr U_{i+1}$. For any vertex $\mathscr B$ in the
orbit of $\mathscr A$, there is an isogeny $\psi : \mathscr A \rightarrow \mathscr B$ of degree coprime to $\ell$. The isogeny $\psi\circ\varphi$ factors
through a variety ${\mathscr B}'$ via an isogeny $ \psi' : {\mathscr A}' \rightarrow {\mathscr B}'$ of same degree as $\psi$, and an
isogeny $\nu : {\mathscr B}' \rightarrow \mathscr B$ of kernel $\psi'(\ker \varphi)$. In particular, $\nu$ is an $\mathfrak l$-isogeny, and $\mathscr B'$ is in the orbit of $\mathscr A'$ for the CM-action, so it is in $\mathscr U_{i}$. This proves that any vertex in the orbit of $\mathscr A$ is reachable by an isogeny down from $\mathscr U_{i}$.

Let $\mathscr E_i$ be the set of all edges (counted with multiplicities) from $\mathscr U_i$ to $\mathscr U_{i+1}$. From Proposition~\ref{prop:frakLStructure}, we have
\begin{equation}\label{eq:Ulysse}
|\mathscr E_i| = \left[(\cO_{i} \otimes_{\Z}\Z_\ell)^\times : ({\cO_{i+1} \otimes_{\Z}\Z_\ell)}^\times\right]\cdot|\mathscr U_i|.
\end{equation}
For any $\mathscr B \in \mathscr U_{i+1}$, let $d(\mathscr B)$ be the number of edges in $\mathscr E_i$ targeting $\mathscr B$ (with multiplicities). We have seen that any $\mathscr B$ is reachable from $\mathscr U_{i}$, therefore $d(\mathscr B) \geq 1$, and we deduce from Lemma~\ref{lemma:distinctFrakLKernels} that $d(\mathscr B) \geq \left[\cO_{i}^\times : \cO_{i+1}^\times\right]$. We deduce
\[|\mathscr E_i| = \sum_{\mathscr B \in \mathscr U_{i+1}} d(\mathscr B) \geq \left[\cO_{i}^\times : \cO_{i+1}^\times\right]\cdot|\mathscr U_{i+1}|.\]
Together with Equation~\eqref{eq:Ulysse}, we obtain the inequality
\begin{equation}\label{eq:Penelope}
|\mathscr U_{i+1}| \leq  \frac{\left[(\cO_{i} \otimes_{\Z}\Z_\ell)^\times : ({\cO_{i+1} \otimes_{\Z}\Z_\ell)}^\times\right]}{\left[\cO_{i}^\times : \cO_{i+1}^\times\right]}\cdot|\mathscr U_i|.
\end{equation}
Since the CM-action of the Picard group of $\cO_i$ is free, we obtain from Lemma~\ref{lemma:classNumberRelation} that the right-hand side of Equation~\eqref{eq:Penelope} is exactly the size of the orbit of any vertex in $\mathscr U_{i+1}$. So $\mathscr U_{i+1}$ contains at most one orbit, and thereby contains exactly one, turning Equation~\eqref{eq:Penelope} into an actual equality.
In particular, all the edges in $\mathscr E_i$ must have multiplicity precisely $[\cO_{i}^\times : \cO_{i+1}^\times]$.
This conclude the recursion.

Note that with all these properties, the graph is a volcano if and only if it is undirected, and all the vertical multiplicities are $1$. The latter is true if and only if $[\cO_{i}^\times : \cO_{i+1}^\times] = 1$ for any $i$, i.e., if $\cO_0^\times \subset K_0$. For the following, suppose it is the case; it remains to decide when the graph is undirected.
If $\mathfrak l$ is principal in $\cO_0 \cap K_0$, the surface $\mathscr V_0$ is undirected because the primes above $\mathfrak l$ in $\cO_0$ are inverses of each other.
If $\varphi: \mathscr A \rightarrow \mathscr B$ is a descending $\mathfrak l$-isogeny within $\mathscr V$, then the unique ascending $\mathfrak l$-isogeny from $\mathscr B$ points to $\mathscr A/\mathscr A[\mathfrak l]$, which is isomorphic to $\mathscr A$ if and only if $\mathfrak l$ is principal in $\cO(\mathscr A)$. So for each descending edge $\mathscr A \rightarrow \mathscr B$ there is an ascending edge $\mathscr B \rightarrow \mathscr A$, and since we have proven above that each vertical edge has multiplicity 1, we conclude that the graph is undirected (so is a volcano) if and only if $\mathfrak l$ is principal in $\cO_0 \cap K_0$ (if $\mathfrak l$ is not principal in $\cO_0 \cap K_0$, there is a level $i$ where $\mathfrak l$ is not principal in $\cO_i$).

For Point~\ref{ascendingImpliesDescending},  choose a descending edge $\mathscr A \rightarrow \mathscr B$. We get that $\mathscr C \cong \mathscr A/\mathscr A[\mathfrak l]$. It is then easy to see that the isogeny $\mathscr A \rightarrow \mathscr B$ induces an isogeny $\mathscr C \rightarrow \mathscr B/\mathscr B[\mathfrak l]$.
\qed\\

Theorem~\ref{thm:lisogenyvolcanoes} gives a complete description of the graph: it allows one to construct an abstract model of any connected component corresponding to an order $\cO_0$ from the knowledge of the norm of $\mathfrak l$, of the (labeled) Cayley graph of the subgroup of $\Pic(\cO_0)$ with generators the prime ideals in $\cO_0$ above $\mathfrak l$, of the order of $\mathfrak l$ in each Picard group $\Pic(\cO_i)$, and of the indices $[\cO_{i}^\times : \cO_{i+1}^\times]$.

\begin{example}
For instance, suppose that $\ell = 2$ ramifies in $K_0$ as $\mathfrak l^2$, and $\mathfrak l$ is principal in $\cO_K$, but is of order 2 in both $\Pic(\cO_{K_0} + \mathfrak l\cO_K)$ and $\Pic(\cO_{K_0} + \mathfrak l^2\cO_K)$, and that $\cO_K^\times \subset K_0$. Then, the first four levels of any connected component of the $\mathfrak l$-isogeny graph for which the largest order is $\cO_K$ are isomorphic to the graph of Figure~\ref{fig:exampleOfNonVolcano}. It is not a volcano since $\mathfrak l$ is not principal in every order $\cO_{K_0} + \mathfrak l^i\cO_K$.
\end{example}

\begin{figure}
\includegraphics{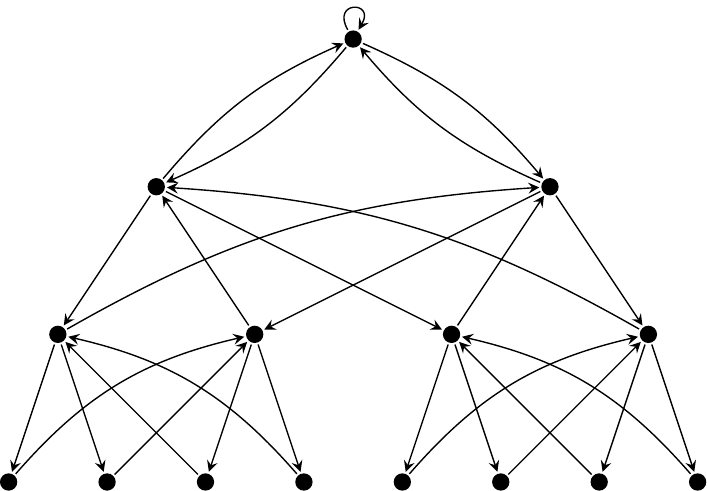}
\caption{\label{fig:exampleOfNonVolcano} An example of an $\mathfrak l$-isogeny graph which is not a volcano, because the ideal $\mathfrak l$ is not principal.}
\end{figure}

\begin{example}
When $K$ is a primitive quartic CM-field, we have seen in Remark~\ref{rem:Streng} that
the multiplicities $[\cO_{i}^\times : \cO_{i+1}^\times]$ are always one, except maybe if
$K = \Q(\zeta_5)$. Actually, even for $K = \Q(\zeta_5)$, only the maximal order $\cO_K$ has
units that are not in $K_0$. We give in Figure~\ref{fig:exampleZeta5} examples of
$\mathfrak l$-isogeny graphs when the order at the surface is $\cO_K = \Z[\zeta_5]$ (which is a principal ideal domain). The primes $2$ and
$3$ are inert in $K$, so we consider $\mathfrak l = 2\cO_{K_0}$ and $\mathfrak l = 3\cO_{K_0}$, and the prime number
$5$ is ramified in $K_0$ so $\mathfrak l^2 = 5\cO_{K_0}$ (and $\mathfrak l$ is also ramified
in $K$, explaining the self-loop at the surface of the last graph).
\end{example}

\begin{figure}
\includegraphics{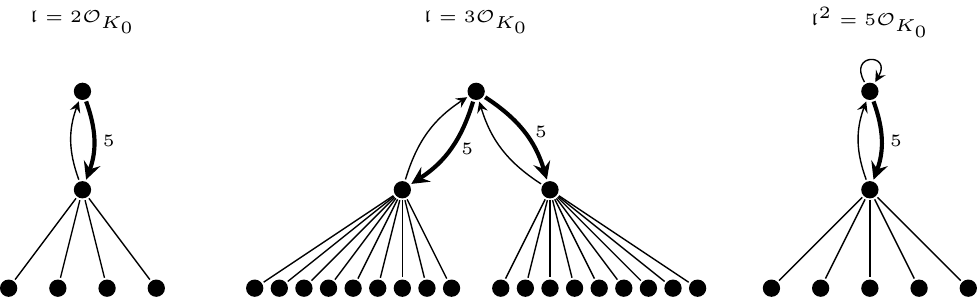}
\caption{\label{fig:exampleZeta5}Some $\mathfrak l$-isogeny graphs for $K = \Q(\zeta_5)$, when the endomorphism ring at the surface is the maximal order $\Z[\zeta_5]$. All edges are simple except the thick ones, of multiplicity 5. The undirected edges are actually directed in both directions.}
\end{figure}

\begin{notation}\label{notation:lVolcano}
Let $\cO$ be any order in $K$ with locally maximal real multiplication at $\ell$, whose conductor is not divisible by $\mathfrak l$. 
We denote by $\mathscr V_\mathfrak l(\cO)$ the connected graph $\mathscr V$ described in Theorem~\ref{thm:lisogenyvolcanoes}.
If $\mathfrak l$ does divide the conductor of $\cO$, let $\cO'$ be the smallest order containing $\cO$, whose order is not divisible by $\mathfrak l$. Then, we also write $\mathscr V_\mathfrak l(\cO)$ for the graph $\mathscr V_\mathfrak l(\cO')$.
\end{notation}

\section{Graphs of $\mathfrak{l}$-isogenies with polarization} \label{sec:Polarizations}
When $\mathfrak l$ is trivial in the narrow class group of $K_0$, then $\mathfrak l$-isogenies preserve principal polarizability. The graphs of $\mathfrak l$-isogenies studied in Section~\ref{subsec:frakLIsogenyGraphs} do not account for polarizations. The present section fills this gap, by describing polarized graphs of $\beta$-isogenies, where $\beta \in K_0$ is a totally positive generator of $\mathfrak l$.
The main result of this section is Theorem \ref{thm:polarizedbetaisogenyvolcanoes} according to which the connected components of polarized isogeny graphs are either isomorphic to the corresponding components of the non-polarized isogeny graphs, or non-trivial double-covers thereof. Yet, this description is not quite exact due to problems arising when the various abelian varieties occurring in a connected component have different automorphism groups.

\subsection{Graphs with polarization}
Before defining the graph, we record the following proposition, which implies that one vertex of a fixed connected component of $(\mathscr W_\beta, v_\beta)$ is principally polarizable if and only if all of them are. Note that since $\beta$ is a generator of $\mathfrak l$, we will write $\beta$-isogeny to mean $\mathfrak l$-isogeny.

\begin{proposition}\label{prop:UniquePolarization}
If $\varphi: \mathscr A \to \mathscr B$ is a $\beta$-isogeny, then there is a unique principal polarization $\xi_{\mathscr B}$ on $\mathscr B$ satisfying
$$
\varphi^* \xi_{\mathscr B} = \xi_{\mathscr A}^\beta.
$$
\end{proposition}

\begin{proof}
Writing $\varphi_{\xi_{\mathscr A}}$ the polarization isogeny, then  $\ker(\varphi) \subset \ker (\varphi_{\xi_{\mathscr A}^\beta})$ is a maximal isotropic subgroup for the commutator pairing and hence by Grothendieck descent (see \cite[Lem.2.4.7]{drobert:thesis}); the proof there is in characteristic $0$, but it extends to ordinary abelian varieties in characteristic $p$ via to canonical lifts), it follows that $\xi_{\mathscr A}^{\beta}$ is a pullback of a principal polarization $\xi_{\mathscr B}$ on $\mathscr B$. For uniqueness, note that the homomorphism $\varphi^* \colon \NS(\mathscr B) \to \NS(\mathscr A)$ of free abelian groups of the same rank becomes an isomorphism after tensoring with $\mathbb{Q}$, hence is injective.
\end{proof}

We define the principally polarized, leveled, $\beta$-isogeny graph $(\mathscr W_\beta^\mathrm{pp}, v_\beta)$ as follows. A point is an isomorphism class\footnote{Recall that two polarizations $\xi_\mathscr A$ and $\xi'_\mathscr A$ on $\mathscr A$ are isomorphic if and only if there is a unit $u \in \cO(\mathscr A)^{\times}$ such that $\xi'_\mathscr A = u^*\xi_\mathscr A$.} of pair $(\mathscr A, \xi_{\mathscr A})$, where $\mathscr A$ is a principally polarizable abelian variety occuring in $(\mathscr W_\beta, v_\beta)$, and $\xi_{\mathscr A}$ is a principal polarization on $\mathscr A$.
There is an edge of multiplicity $m$ from the isomorphism class of $(\mathscr A, \xi_{\mathscr A})$ to the isomorphism class of $(\mathscr B, \xi_{\mathscr B})$ if there are $m$ distinct subgroups of $\mathscr A$ that are kernels of $\beta$-isogenies $\varphi : \mathscr A \rightarrow \mathscr B$ such that $\varphi^*\xi'_{\mathscr B}$ is isomorphic to $\xi_{\mathscr A}^\beta$, for some polarization $\xi'_{\mathscr B}$ isomorphic to $\xi_{\mathscr B}$.
The graph $\mathscr W_\beta^\mathrm{pp}$ admits a forgetful map to $\mathscr W_\beta$, and in particular inherits the structure of a leveled graph $(\mathscr W_\beta^\mathrm{pp}, v_\beta)$.

\begin{remark}
It can be the case that there is no $\beta$-isogeny $\varphi : \mathscr A \to \mathscr B$ such that $\varphi^* \xi_{\mathscr{B}} \cong \xi_{\mathscr A}^\beta$, but that there is nonetheless an edge (because there is a map with this property for some other polarization $\xi'_{\mathscr B}$, isomorphic to $\xi_{\mathscr B}$). This can happen because pullbacks of isomorphic polarizations are not necessarily isomorphic, when $\mathscr A$ and $\mathscr B$ have different automorphism groups.
\end{remark}

We note that this graph is undirected:
\begin{proposition}\label{prop:BetaDual}
If $\varphi: \mathscr A \to \mathscr B$ is a $\beta$-isogeny, then there is a unique $\beta$-isogeny $\tilde\varphi: \mathscr B \to \mathscr A$ satisfying $\tilde\varphi \varphi= \beta$, called the $\beta$-dual of $\varphi$.
\end{proposition}

\begin{proof}
Let $\kappa$ be the kernel of $\varphi$. 
The group $\mathscr A[\beta]$ is an $\cO_0(\mathscr A)/(\beta)$-vector space of dimension 2, of which the kernel $\kappa$ is a vector subspace of dimension 1. Therefore there is another vector subspace $\kappa'$ such that $\mathscr A[\beta] = \kappa \oplus \kappa'$, and
$\varphi(\kappa')$ is the kernel of a $\beta$-isogeny
$\psi : \mathscr B \rightarrow \mathscr C$.
Then, the kernel of the composition $\psi \circ \varphi$ is $\mathscr A[\beta]$ so there is an isomorphism $u : \mathscr C \ra \mathscr A$ such that $u\circ \psi \circ \varphi = \beta$.
The isogeny $u\circ \psi$ is the $\beta$-dual of $\varphi$ (which is trivially unique).
\end{proof}

\subsection{Counting polarizations}
To describe $(\mathscr W_\beta^\mathrm{pp}, v_\beta)$, we need to count principal polarizations on any fixed variety. If $\mathcal{O}$ is an order in $K$, write $\mathcal{O}^{+\times}$ for the group of totally positive units in $\mathcal{O} \cap K_0$.
\begin{proposition}\label{allPolarizations}
Let $\mathscr A$ be a simple ordinary abelian variety over $\mathbb{F}_q$ with endomorphism ring $\mathcal{O}$. Then the set of isomorphism classes of principal polarizations (when non-empty) on $\mathscr A$ is a torsor for the group
$$
U(\mathcal{O}) := \frac{\mathcal{O}^{+\times}}{\mathbf{N}: \mathcal{O}^\times \to (\mathcal{O} \cap K_0)^\times}.
$$
\end{proposition}

\begin{proof}
See \cite[Cor.5.2.7]{BL04} for a proof in characteristic $0$. That the result remains true for ordinary abelian varieties in characteristic $p$ follows from the theory of canonical lifts.
\end{proof}

The following lemma recalls some well-known facts about $U(\mathcal{O})$. 

\begin{lemma}\label{lem:units}
The group $U(\mathcal{O})$ is an $\mathbb{F}_2$-vector space of dimension $d$, where $0 \leq d \leq g-1$. If $\mathcal{O} \subset \mathcal{O'}$ and $\mathcal{O} \cap K_0 = \cO'  \cap K_0$, then the natural map $U(\mathcal{O}) \to U(\mathcal{O'})$ is surjective.
\end{lemma}

\begin{proof}
Writing $\mathbf{N}$ for the norm from $K$ to $K_0$, we have the following hierarchy, the last containment following because for $\beta \in \mathcal{O}_r$ one has $\beta^2 = \mathbf N \beta$:
\begin{equation}\label{hierarchy}
(\mathcal{O} \cap K_0)^\times \supseteq \mathcal{O}^{+\times} \supseteq \mathbf{N}(\mathcal{O}^\times)  \supseteq (\mathcal{O} \cap K_0)^{\times 2}
\end{equation}
By Dirichlet's unit theorem (and its extension to non-maximal orders), the group $\mathcal{O}_0^\times$ is of the form $\{\pm 1\} \times A$, where $A$ is a free abelian group of cardinality $2^{g-1}$, so the quotient 
$(\mathcal{O} \cap K_0)^\times/(\mathcal{O} \cap K_0)^{\times 2}$ is an $\mathbb{F}_2$-vector space of dimension at most $g$. Since $-1$ is never a totally positive unit, the first claim follows. The second sentence of the lemma is clear.
\end{proof}

\begin{remark}
We remark that, other than the simple calculations described, there is little one can say in great generality about the indices of the containments in (\ref{hierarchy}), which vary depending on the specific fields $K$ and orders $\mathcal{O}$ chosen. For example, if $g = 2$, the total index in (\ref{hierarchy}) is $4$, and one has examples with the ``missing'' factor of $2$ (i.e., the one unaccounted for by the totally negative unit $-1$) occurring in any of the three containments. 
\end{remark}

\subsection{Structure of $(\mathscr W_\beta^\mathrm{pp}, v_\beta)$}
We may now state the main theorem.

\begin{theorem}\label{thm:polarizedbetaisogenyvolcanoes}
Let $\mathscr V^\mathrm{pp}$ be any connected component of the leveled $\beta$-isogeny graph $(\mathscr W_\beta^\mathrm{pp}, v_\beta)$.
For each $i \geq 0$, let $\mathscr V^\mathrm{pp}_i$ be the subgraph of $\mathscr V^\mathrm{pp}$ at level $i$.
We have:
\begin{enumerate}[label=(\roman*)]
\item \label{thmitem:polLevels} For each $i \geq 0$, the varieties in $\mathscr V^\mathrm{pp}_i$ share a common endomorphism ring $\cO_i$. The order $\cO_0$ can be any order with locally maximal real multiplication at $\ell$, whose conductor is not divisible by $\beta$;
\item \label{thmitem:polLevel0}The level $\mathscr V^\mathrm{pp}_0$ is isomorphic to the Cayley graph of the subgroup of $\mathfrak C(\cO_0)$ with generators $(\mathfrak L_i, \beta)$ where $\mathfrak L_i$ are the prime ideals in $\cO_0$ above $\beta$;
\item For any $\mathscr A \in \mathscr V^\mathrm{pp}_0$, there are 
$$\frac {N(\mathfrak l)-\left(\frac{K}{\beta}\right)}{[\cO_{0}^\times : \cO_{1}^\times]}\frac{U(\cO_{1})}{U(\cO_{0})}$$
 edges of multiplicity $[\cO_{0}^\times : \cO_{1}^\times]$ from $\mathscr A$ to distinct vertices of~$\mathscr V^\mathrm{pp}_{1}$ (where $\left(\frac{K}{\beta}\right)$ is $-1$, $0$ or $1$ if $\beta$ is inert, ramified, or split in $K$);
\item For each $i > 0$, and any $x \in \mathscr V^\mathrm{pp}_i$, there is one simple edge from $x$ to a vertex of $\mathscr V^{pp}_{i-1}$, and
$$\frac {N(\mathfrak l)}{[\cO_{i}^\times : \cO_{i+1}^\times]}\frac{U(\cO_{i+1})}{U(\cO_{i})}$$
edges of multiplicity $[\cO_{i}^\times : \cO_{i+1}^\times]$ to distinct vertices of $\mathscr V^\mathrm{pp}_{i+1}$;
\item\label{thmitem:almostundirected} For each edge $x \rightarrow y$, there is an edge $y \rightarrow x$.
\end{enumerate}
In particular, the graph $\mathscr V^\mathrm{pp}$ is an $N(\beta)$-volcano if and only if $\cO_0^\times \subset K_0$.
Also, if $\mathscr V^\mathrm{pp}$ contains a variety defined over the finite field $k$, the subgraph containing only the varieties defined over $k$ consists of the subgraph of the first $v$ levels, where $v$ is the valuation at $\beta$ of the conductor of $\cO_{K_0}[\pi] = \cO_{K_0}[\pi, \pi^\dagger]$.
\end{theorem}

Before proving this theorem, we need some preliminary results.
First, we recall the action of the Shimura class group. For $\mathcal{O}$ an order, write $\mathscr{I}(\mathcal{O})$ for the group of invertible $\mathcal{O}$-ideals, and define the Shimura class group as
$$
\mathfrak{C}(\mathcal{O}) = \{ (\mathfrak{a}, \alpha) \mid \mathfrak{a} \in \mathscr{I}(\mathcal{O}): \mathbf N\mathfrak{a} = \alpha\mathcal{O}, \alpha \in K_0 \text{ totally positive } \} /\sim
$$
where two pairs $(\mathfrak{a}, \alpha), (\mathfrak{a}', \alpha')$ are equivalent if there exists $u \in K^\times$ with $\mathfrak{a}' = u\mathfrak{a}$ and $\alpha' = uu^\dagger \alpha$. 
The Shimura class group acts freely on the set of isomorphism classes of principally polarized abelian varieties whose endomorphism ring is $\mathcal{O}$ (see~\cite[\S 17]{taniyama-shimura} for the result in characteristic 0, which extends via canonical lifts to the ordinary characteristic $p$ case). If $\beta$ is coprime to the conductor of $\mathcal{O}$, then an element of $\mathfrak C(\cO)$ acts by a $\beta$-isogeny if and only if it is of the form $(\mathfrak{L}, \beta)$, for some prime ideal $\mathfrak{L}$ of $\cO$ dividing $(\beta)$.

\begin{lemma}\label{lemma:uniquePolUp}
Let $\varphi : \mathscr A \rightarrow \mathscr B$ be a $\beta$-isogeny, and let $\xi_\mathscr A$ be a principal polarization on $\mathscr A$. We have:
\begin{enumerate}[label=(\roman*)]
\item \label{lemmaitem:ascendingpol} If $\varphi$ is $\beta$-ascending, there is, up to isomorphism, a unique polarization $\xi_\mathscr B$ on $\mathscr B$ such that $\varphi^*\xi_\mathscr B$ is isomorphic to $\xi_\mathscr A^\beta$;
\item \label{lemmaitem:descendingpol} It $\varphi$ is $\beta$-descending, there are, up to isomorphism, exactly 
\[\frac {|U(\mathcal{O}(\mathscr B))|}{|U(\mathcal{O}(\mathscr A))|}\]
distinct polarizations $\xi_\mathscr B$ on $\mathscr B$ such that $\varphi^*\xi_\mathscr B$ is isomorphic to $\xi_\mathscr A^\beta$.
\end{enumerate}
\end{lemma}

\begin{proof}
Let us first prove \ref{lemmaitem:ascendingpol}.
From Proposition~\ref{prop:UniquePolarization}, there exists a polarization $\xi_\mathscr B$ on $\mathscr B$ such that $\varphi^*\xi_\mathscr B = \xi_\mathscr A^\beta$.
Suppose $\xi'_\mathscr B$ is a polarization such that $\varphi^*\xi'_\mathscr B \cong \xi_\mathscr A^\beta$. Then, there is a unit $u \in \cO(\mathscr A)^{\times}$ such that $\varphi^*\xi'_\mathscr B = u^*\xi_\mathscr A^{\beta}$. But $\varphi$ is ascending, so $u \in \cO(\mathscr B)^{\times}$ and therefore $$\varphi^*\xi'_\mathscr B = u^*\xi_\mathscr A^\beta = u^*(\varphi^*\xi_\mathscr B) = \varphi^*(u^*\xi_\mathscr B).$$
From the uniqueness in Proposition~\ref{prop:UniquePolarization}, we obtain $\xi'_\mathscr B = u^*\xi_\mathscr B$, so $\xi_\mathscr B$ and $\xi'_\mathscr B$ are two isomorphic polarizations.

For \ref{lemmaitem:descendingpol}, again apply Proposition~\ref{prop:UniquePolarization}, and observe that the kernel of the surjection
$U(\mathcal{O}(\mathscr B)) \ra U(\mathcal{O}(\mathscr A))$
of Lemma~\ref{lem:units} acts simply transitively on the set of isomorphism classes of polarizations $\xi_\mathscr B$ on $\mathscr B$ satisfying $\varphi^*\xi_\mathscr B \cong\xi_\mathscr A^\beta$.
\end{proof}

\subsection*{Proof of Theorem \ref{thm:polarizedbetaisogenyvolcanoes}}
First observe that \ref{thmitem:polLevels} is immediate from Theorem~\ref{thm:lisogenyvolcanoes}(i), since the leveling on $\mathscr{V}^{\mathrm{pp}}$ is induced from that of $\mathscr{V}$. Also, \ref{thmitem:almostundirected} is a direct consequence of the existence of $\beta$-duals, established in Proposition~\ref{prop:BetaDual}.
Now, let us prove that for any class $(\mathscr A, \xi_\mathscr A)$ at a level $i > 0$, there is a unique edge to the level $i-1$.
From Theorem~\ref{thm:lisogenyvolcanoes}, there exists an ascending isogeny $\varphi : \mathscr A \rightarrow \mathscr B$ (unique up to isomorphism of $\mathscr B$), and from 
Lemma~\ref{lemma:uniquePolUp}\ref{lemmaitem:ascendingpol}, there is a unique polarization $\xi_\mathscr B$ on $\mathscr B$ (up to isomorphism) such that $(\mathscr A, \xi_\mathscr A) \ra (\mathscr B, \xi_\mathscr B)$ is an edge in $\mathscr{V}^{\mathrm{pp}}$.

These results, and the fact that $\mathscr V_0$ is connected, imply that $\mathscr V_0^\mathrm{pp}$ is connected. We can then deduce \ref{thmitem:polLevel0} from the action of the Shimura class group $\mathfrak C(\cO_0)$.

Now, (iii) (respectively, (iv)) is a consequence of Theorem~\ref{thm:lisogenyvolcanoes}(iii) (respectively, Theorem~\ref{thm:lisogenyvolcanoes}(iv)) together with Lemma~\ref{lemma:uniquePolUp}. The statement on multiplicities of the edges also uses the fact that if $\varphi,\psi : \mathscr A\ra \mathscr B$ are two $\beta$-isogenies with same kernel, and $\xi_\mathscr A$ is a principal polarization on $\mathscr A$, then the two principal polarizations on $\mathscr B$ induced via $\varphi$ and $\psi$ are isomorphic.

The volcano property follows from the corresponding phrase in the statement of Theorem \ref{thm:lisogenyvolcanoes}, and the statement on fields of definition follows from Remark \ref{rem:fieldOfDefIsogenies}, which shows that the isomorphism from a principally polarized absolutely simple ordinary abelian variety to its dual, and hence the polarization, is defined over the field of definition of the variety.
\qed

\subsection{Principally polarizable surfaces}\label{subsec:ppas}
The result of Theorem~\ref{thm:polarizedbetaisogenyvolcanoes} for abelian surfaces is a bit simpler than the general case, thanks to the following lemma.

\begin{lemma}
Suppose $g = 2$.
With all notations as in Theorem~\ref{thm:polarizedbetaisogenyvolcanoes}, we have $U(\mathcal{O}_i) = U(\mathcal{O}_0)$ for any non-negative integer $i$.
\end{lemma}
\begin{proof}
In these cases, one has $\mathcal{O}_K^\times = \mathcal{O}_{K_0}^\times$ except in the case $K = \Q(\zeta_5)$ (see Remark \ref{rem:Streng}); but even when $K = \Q(\zeta_5)$ the equality is true up to units of norm~$1$.  Therefore for any order $\cO$ in $K$, one has $ \mathbf N \mathcal{O}^\times = \mathbf N(\cO \cap {K_0})^\times$. Thus, none of the groups $U(\mathcal{O}_i)$ actually depend on $i$.
\end{proof}

Therefore, the factors $|U(\mathcal{O}_{i+1})|/|U(\mathcal{O}_i)|$ disappear when $g = 2$.
It follows that each component $\mathscr{W}^\mathrm{pp}$ is either isomorphic to its image in $(\mathscr W_\beta, v_\beta)$, or is isomorphic to the natural double cover of this image constructed by doubling the length of the cycle $\mathscr V_0$ (as illustrated in Figure~\ref{fig:examplePlusMinusVolcano}).
The first case occurs when $(\beta)$ is inert in $K/K_0$, or when the order of $(\mathfrak L, \beta)$ in $\mathfrak C(\cO_0)$ equals the order of $\mathfrak L$ in $\Cl(\cO_0)$ (where $\mathfrak L$ is a prime ideal of $\cO_0$ above $(\beta)$). The second case occurs when the order of $(\mathfrak L, \beta)$ is twice that of $\mathfrak L$.

\begin{figure}
\includegraphics{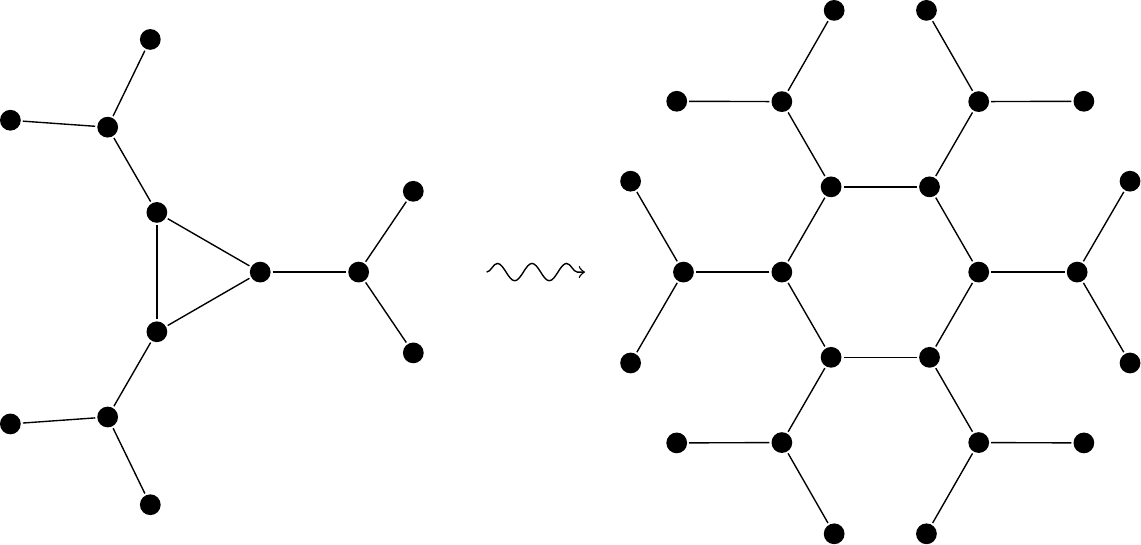}
\caption{\label{fig:examplePlusMinusVolcano} An example of how adding the polarization data to a volcano of $\beta$-isogenies can double the length of the cycle.}
\end{figure}

\section{Levels for the real multiplication in dimension 2}\label{sec:levelsRM}

We now specialize to the case $g = 2$. Then, $\mathscr A$ is of dimension 2, and $K$ is a primitive quartic CM-field. The subfield $K_0$ is a real quadratic number field. The orders in $K_{0,\ell}$ are linearly ordered since they are all of the form $\Z_\ell + \ell^n\lO_0$. These $n$'s can be seen as ``levels'' of real multiplication. Taking advantage of this simple structure, the goal of this section is to prove Theorem~\ref{RMupIso}.

\subsection{Preliminaries on symplectic lattices} Let $\F_\ell$ be the finite field with $\ell$ elements. 

\begin{lemma}\label{lemma:countingMaxIso}
Let $W$ be a symplectic $\F_\ell$-vector space of dimension $4$. 
It contains exactly $\ell^{3} + \ell^{2} + \ell + 1$ maximal isotropic subspaces.
\end{lemma}

\begin{proof}
In the following, a \emph{line} or a \emph{plane} means a dimension 1 or 2 subspace of a vector space (i.e., they contain the origin of the vector space).
Fix any line $L$ in $W$. We will count the number of maximal isotropic subspaces of $W$ containing $L$. The line $L$ is itself isotropic (yet not maximal), so $L \subset L^\perp$. Also, $\dim L + \dim L^\perp = 4$, so $\dim L^\perp = 3$. Since any maximal isotropic subspace of $W$ is of dimension 2, it is easy to see that those containing $L$ are exactly the planes in $L^\perp$ containing $L$. There are $\ell + 1$ such planes, because they are in natural correspondence with the lines in the dimension 2 vector space $L^\perp / L$. It follows that there are $\ell + 1$ maximal isotropic subspaces of $W$ containing $L$. There are $\ell^{3} + \ell^{2} + \ell + 1$ lines $L$ in $W$, and each maximal isotropic subspace of $W$ contains $\ell + 1$ lines, we conclude that there are $\ell^{3} + \ell^{2} + \ell + 1$ maximal isotropic subspaces.
\end{proof}

\begin{lemma}
Let $V$ be a symplectic $\Q_\ell$-vector space of dimension 4. Let $\Lambda \subset V$ be a lattice in $V$ such that $\Lambda^* = \Lambda$. Then $\Lambda/\ell \Lambda$ is a symplectic $\F_\ell$-vector space of dimension $4$ for the symplectic form
$$\langle \lambda + \ell\Lambda, \mu + \ell\Lambda \rangle_\ell = \langle \lambda , \mu \rangle \mod \ell.$$
\end{lemma}

\begin{proof}
The fact that the form $\langle -, - \rangle_\ell$ is bilinear and alternating easily follows from the fact that the form $\langle -, - \rangle$ is symplectic. It only remains to prove that it is non-degenerate. Let $\lambda \in \Lambda$, and suppose that $\langle \lambda + \ell\Lambda, \mu + \ell\Lambda \rangle_\ell = 0$ for any $\mu \in \Lambda$. We now prove that $\lambda \in \ell\Lambda$. For any $\mu \in \Lambda$, we have $\langle \lambda , \mu \rangle \in \ell \Z_\ell$, and therefore
$\langle \ell^{-1}\lambda , \mu \rangle \in \Z_\ell.$
So $\ell^{-1}\lambda \in \Lambda^* = \Lambda$, whence $\lambda \in \ell \Lambda$, concluding the proof.
\end{proof}

\begin{lemma}
Let $V$ be a symplectic $\Q_\ell$-vector space of dimension 4, and $\Lambda$ a self-dual lattice in $V$.
Let $\ell\Lambda \subset \Gamma \subset \Lambda$ be an intermediate lattice. Then $\Gamma/\ell\Lambda$ is maximal isotropic in $\Lambda/\ell \Lambda$ if and only if $\Gamma^* = \ell^{-1}\Gamma$.
\end{lemma}

\begin{proof}
First, suppose that $\Gamma/\ell\Lambda$ is maximal isotropic. Fix $\gamma \in \Gamma$. For any $\delta \in \Gamma$, since $\Gamma/\ell\Lambda$ is isotropic, we have $\langle \gamma, \delta \rangle \in \ell \Z_\ell$, so $\langle \ell^{-1}\gamma, \delta \rangle \in \Z_\ell$ and therefore $\ell^{-1}\gamma \in \Gamma^*$. This proves that $\ell^{-1}\Gamma \subset \Gamma^*$. Now, let $\alpha \in \Gamma^*$.
Observe that $\langle \ell\alpha, \gamma \rangle = \ell\langle \alpha, \gamma \rangle \in \ell\Z_\ell$ for any $\gamma \in \Gamma$. This implies that $\ell^{-1}\alpha$ must be in $\Gamma$, because $\Gamma/\ell\Lambda$ is maximally isotropic. This proves that $\ell^{-1}\Gamma^* \subset \Gamma$.

Now, suppose that $\Gamma^* = \ell^{-1}\Gamma$. 
Then, $\langle  \ell^{-1} \Gamma, \Gamma \rangle \subset \Z_\ell$, so $\langle  \Gamma, \Gamma \rangle \in \ell\Z_\ell$, and $\Gamma/\ell \Lambda$ is isotropic. Let $\lambda \in \Lambda$ such that $\langle \lambda + \ell \Lambda, \Gamma/\ell \Lambda \rangle_\ell = \{0\}$. Then, $\langle \ell^{-1}\lambda , \Gamma \rangle \subset \ell\Z_\ell$, so $\ell^{-1}\lambda \in \Gamma^* = \ell^{-1}\Gamma$, which implies that $\lambda \in \Gamma$. So $\Gamma/\ell \Lambda$ is maximal isotropic.
\end{proof}

\begin{definition}[($\ell,\ell$)-neighbors]
The set $\mathscr L(\Lambda)$ of \emph{$(\ell,\ell)$-neighbors} of $\Lambda$ is the set of lattices $\Gamma$ such that $\ell \Lambda \subset \Gamma \subset \Lambda$, and $\Gamma/\ell\Lambda$ is maximal isotropic in $\Lambda/\ell\Lambda$.
\end{definition}

\begin{remark}\label{rem:correspEllEllAndEllEll}
Consider the lattice $T = T_\ell\mathscr A$.
Note that $(\ell, \ell)$-isogenies $\mathscr A \to \mathscr B$ correspond under Proposition \ref{prop:correspondence} to lattices $\Gamma$ with $T \subset \Gamma \subset \frac{1}{\ell} T$ and $\Gamma/T$ a maximal isotropic subspace of $\frac{1}{\ell} T / T$, i.e., to $\ll$-neighbors of $T$ rescaled by a factor $\ell^{-1}$.
\end{remark}

\subsection{$\ll$-neighboring lattices}
Throughout this section, $V$ is a symplectic $\Q_\ell$-vector space of dimension 4. Again, we consider a prime number $\ell$, a quartic CM-field $K$, with $K_{0}$ its quadratic real subfield. The algebra $K_\ell = K\otimes_{\Q}\Q_\ell$ is a $\Q_\ell$-algebra of dimension~4, with an involution $x \mapsto x^\dagger$ fixing $K_{0,\ell}$ induced by the generator of $\Gal(K/K_0)$. Suppose that $K_\ell$ acts ($\Q_\ell$-linearly) on~$V$, and that for any $x \in K_\ell$, $u,v \in V$, we have $\langle xu,v \rangle = \langle u,x^\dagger v \rangle$.
For any lattice $\Lambda$ in $V$, the \emph{real order} of $\Lambda$ is the order in $K_{0,\ell} = K_0\otimes_{\Q}\Q_\ell$ defined as
$$\lO_0(\Lambda) = \{x \in K_{0,\ell} \mid x\Lambda \subset \Lambda \}.$$
Any order in $K_{0,\ell}$ is of the form $\lO_n = \Z_\ell + \ell^{n}\lO_0$, for some non-negative integer $n$, with $\lO_0$ the maximal order of $K_{0,\ell}$.
We say that $\Lambda$ is an $\lO_n$-lattice if $\lO(\Lambda) = \lO_n$. The goal of this section is to prove Theorem~\ref{RMupIso} by first proving its lattice counterpart, in the form of the following proposition.

\begin{proposition}\label{prop:neighborsLevelsRM}
Let $\Lambda$ be a self-dual $\lO_n$-lattice, with $n > 0$. The set $\mathscr L(\Lambda)$ of its $(\ell,\ell)$-neighbors contains exactly one $\lO_{n-1}$-lattice, namely $\ell \lO_{n-1}\Lambda$, $\ell^2 + \ell$ lattices of real order $\lO_{n}$, and $\ell^3$ lattices of real order $\lO_{n+1}$.
\end{proposition}

\begin{lemma}\label{lemma:splittingOnLattice}
Let $\Lambda$ be a self-dual $\lO_n$-lattice in $V$, for some non-negative integer $n$. Then, $\Lambda$ is a free $\lO_n$-module of rank 2.
\end{lemma}

\begin{proof}
By Lemma \ref{lemma:quadraticImpliesGorenstein}, the order $\lO_n$ is a Gorenstein ring of dimension 1, and it follows from~\cite[Thm. 6.2]{Bass63} that $\Lambda$ is a reflexive $\lO_n$-module. From~\cite[Prop. 7.2]{Bass63}, $\Lambda$ has a projective direct summand, so $\Lambda =  \lO_n e_1 \oplus M$ for some $e_1 \in \Lambda$, and $M$ an $\lO_n$-submodule. This $M$ is still reflexive (any direct summand of a reflexive module is reflexive). So applying \cite[Prop. 7.2]{Bass63} again to $M$, together with the fact that it has $\Z_\ell$-rank 2, there is a non-negative integer $m \leq n$ and an element $e_2 \in \Lambda$ such that $M = \lO_me_2$. We shall prove that $m = n$. By contradiction, assume $m < n$. We have $\Lambda/\ell\Lambda = ( \lO_n e_1/\ell\lO_n) \oplus ( \lO_m e_2 / \ell \lO_m)$. Observe that $ \lO_m e_2/ \ell \lO_m$ is maximal isotropic. Indeed, it is of dimension 2, and for any $x,y \in \lO_m$, $\langle x e_2, y e_2 \rangle = - \langle y e_2, x e_2 \rangle$ because the form is alternating, and $\langle x e_2, y e_2 \rangle = \langle y e_2, x e_2 \rangle$ because it is $K_0$-bilinear, so $\langle x e_2, y e_2 \rangle = 0$.
Also, we have $\lO_{n-1} \subset \lO_m$, so 
$$\langle \ell\lO_{n-1} e_1, \lO_m e_2 \rangle = \langle \ell e_1, \lO_{n-1}\lO_m e_2 \rangle = \ell \langle e_1, \lO_m e_2 \rangle \subset \ell \Z_\ell.$$
This proves that $\ell\lO_{n-1}e_1/\ell\lO_n \subset (\lO_m e_2/\ell\lO_m)^\perp = \lO_m e_2/\ell\lO_m$, a contradiction.
\end{proof}

Using a standard abuse of notation, write $\F_\ell[\epsilon]$ for the ring of dual numbers, i.e. an $\F_\ell$-algebra isomorphic to $\F_\ell[X]/X^2$ via an isomorphism sending $\epsilon$ to $X$.
\begin{lemma}\label{lemma:subspacesEpsilon}
Let $R = \F_\ell[\epsilon]f_1 \oplus \F_\ell[\epsilon]f_2$ be a free $\F_\ell[\epsilon]$-module of rank 2.
The $\F_\ell[\epsilon]$-submodules of $R$ of $\F_{\ell}$-dimension 2 are exactly the $\ell^2 + \ell + 1$ modules $\epsilon R$, and $\F_\ell[\epsilon] \cdot g$ for any $g \not \in \epsilon R$. A complete list of these orbits $\F_\ell[\epsilon] \cdot g$ is given by $\F_\ell[\epsilon]\cdot(b\epsilon f_1 + f_2)$ for any $b \in \F_\ell$, and $\F_\ell[\epsilon]\cdot(f_1 + \alpha f_2 + \beta \epsilon f_2)$, for any $\alpha,\beta \in \F_\ell$.
\end{lemma}

\begin{proof}
Let $H \subset R$ be a subspace of dimension 2, stable under the action of $\F_\ell[\epsilon]$. 
For any $g \in H$, write $g = a_gf_1+b_g\epsilon f_1 + c_gf_2 + d_g \epsilon f_2 \in H$ for $a_g,b_g,c_g,d_g \in \F_\ell$.
Since $H$ is $\F_\ell[\epsilon]$-stable, for any $g \in H$, the element $g\epsilon =  a_g\epsilon f_1 + c_g\epsilon f_2$ is also in $H$.

First suppose $a_g = 0$ and $c_g = 0$ for any $g \in H$. Then, as $H = \epsilon R$, it is indeed an $\F_\ell[\epsilon]$-submodule and has $\F_{\ell}$-dimension 2.
Now, suppose $a_g = 0$ for any $g \in H$, but $H$ contains an element $g$ such that $c_g \neq 0$ is non-zero. Then, $H$ contains both $b_g\epsilon f_1 + c_gf_2 + d_g \epsilon f_2$, and $c_g\epsilon f_2$, so $H$ is the $\F_\ell$-vector space spanned by $\epsilon f_2$ and $b_g\epsilon f_1 + c_gf_2$. There are $\ell + 1$ such subspaces $H$ (one for each possible $(b_g:c_g) \in \mathbb P^1(\F_\ell)$), and all of them are of dimension 2 and $R$-stable. 

Finally, suppose there exists $g \in H$ such that $a_g \neq 0$. Then, it is spanned as an $\F_\ell$-vector spaces by a pair $\{f_1 + \alpha f_2 + \beta \epsilon f_2, \epsilon f_1 + \alpha \epsilon f_2\}$, with $\alpha,\beta \in \F_\ell$, and any of the $\ell^2$ subspaces of this form are $\F_\ell[\epsilon]$-submodules.
\end{proof}

\begin{lemma}\label{lemma:realOrbitIsotropic}
Let $\Lambda$ be an $\lO_n$-lattice, for some non-negative integer $n$. For any element $g \in \Lambda/\ell\Lambda$, the orbit $\lO_n \cdot g$ is an isotropic subspace of $\Lambda/\ell\Lambda$.
\end{lemma}

\begin{proof}
Let $\lambda \in \Lambda$ such that $g = \lambda + \ell\Lambda$.
For any $\alpha, \beta \in \lO_n$, we have
$\langle \alpha \lambda, \beta \lambda \rangle = -\langle \beta \lambda,  \alpha \lambda \rangle$
because the symplectic form on $V$ is alternating, and 
$\langle \alpha \lambda, \beta \lambda \rangle = \langle \beta \lambda,  \alpha \lambda \rangle$
because it is $K_0$-bilinear. So $\langle \alpha g, \beta g \rangle_\ell = 0$, and the orbit of $g$ is isotropic.
\end{proof}

\subsection*{Proof of Proposition~\ref{prop:neighborsLevelsRM}}
From Lemma~\ref{lemma:splittingOnLattice}, $\Lambda$ splits as $e_1 \lO_n \oplus e_2 \lO_n$, for some $e_1,e_2 \in \Lambda$.
Observe that there is an element $\epsilon \in \lO_n$ such that $\lO_n / \ell \lO_n = \F_\ell[\epsilon] \cong \F_\ell[X]/(X^2)$, via the isomorphism sending $\epsilon$ to $X$. The quotient
$R = \Lambda / \ell \Lambda$ is a free $\F_\ell[\epsilon]$-module of rank 2. Let $\pi : \Lambda \rightarrow R$ be the
canonical projection. The set $\{f_1,\epsilon f_1, f_2, \epsilon f_2\}$ forms an $\F_\ell$-basis of $R$, where $f_i = \pi(e_i)$.

From Lemma~\ref{lemma:subspacesEpsilon}, $R$ contains $\ell^2 + \ell + 1$ subspaces of dimension 2 that are $\F_\ell[\epsilon]$-stable.
The subspace $\epsilon R = \F_\ell\epsilon f_1 \oplus \F_\ell\epsilon f_2$ is isotropic because
$$\langle \epsilon f_1, \epsilon f_2 \rangle_\ell = \langle f_1, \epsilon^2 f_2 \rangle_\ell = 0.$$
Together with Lemma~\ref{lemma:realOrbitIsotropic}, we conclude that all $\ell^2 + \ell + 1$ of these $\F_\ell[\epsilon]$-stable subspaces are maximal isotropic.
From Lemma~\ref{lemma:countingMaxIso}, $R$ contains a total of $\ell^{3} + \ell^{2} + \ell + 1$ maximal isotropic subspaces.
Thus, the $(\ell,\ell)$-neighbors corresponding to the remaining $\ell^3$ subspaces are not stable for the action of $\lO_n$. They are
however stable for the action of $\lO_{n+1}$, so those are $\lO_{n+1}$-lattices. 

It remains to prove that among the $\ell^{2} + \ell + 1$ neighbors that are $\lO_n$-stable, only the lattice $\ell \lO_{n-1}\Lambda$ (which
corresponds to the subspace $\epsilon R$) is $\lO_{n-1}$-stable, and that it is not $\lO_{n-2}$-stable. This would prove that
$\ell \lO_{n-1}\Lambda$ is an $\lO_{n-1}$-lattice, and the $\ell^2 + \ell$ other lattices have order $\lO_{n}$.

Write $\Gamma = \ell \lO_{n-1}\Lambda$. Then $\pi(\Gamma) = \epsilon R$ is maximal isotropic and $\F_\ell[\epsilon]$-stable.
Suppose by contradiction that we have 
$\lO_{n-2} \Gamma \subset \Gamma$. Then,
$\ell \lO_{n-2}\Lambda \subset \lO_{n-2}\Gamma \subset \Gamma \subset \Lambda,$
so $\ell \lO_{n-2}\Lambda \subset \Lambda$. But $\ell \lO_{n-2} \not \subset \lO_{n}$, which contradicts the fact that $\Lambda$ is an $\lO_n$-lattice. Therefore $\Gamma$ is an $\lO_{n-1}$-lattice.

Let $H \subset R$ be another maximal isotropic subspace, and suppose that $\pi^{-1}(H)$ is $\lO_{n-1}$-stable.
Let $\lambda = e_1(a+\ell^nx) + e_2(b+\ell^ny) \in \pi^{-1}(H)$, with $a,b \in \Z_\ell$ and $x,y \in \lO_0$, and let $z \in \lO_{n-1}$. A simple computation yields
$$\Lambda = z \lambda + \Lambda = z ae_1 + z be_2 + \Lambda.$$
Therefore, both $z a$ and $z b$ must be in $\lO_n$ for any $z \in \lO_{n-1}$. It follows that $a$ and $b$ must be in $\ell\Z_\ell$, whence
$\lambda \in \Gamma$. So $\pi^{-1}(H) \subset \Gamma$, and we conclude that $H = \epsilon R$ from the fact that both are maximal
isotropic. This proves that no $(\ell,\ell)$-neighbor other that $\Gamma$ is $ \lO_{n-1}$-stable.
\qed

\subsection{Changing the real multiplication with $\ll$-isogenies}\label{subsec:ellellLevelsRM}
The results for lattices are now ready to be applied to analyze how $\ll$-isogenies can change the real multiplication.
Fix a principally polarizable absolutely simple ordinary abelian surface $\mathscr A$ over $\F_q$. As usual, $K$ is its endomorphism algebra, and $K_0$ the maximal real subfield of $K$. 
The local real order $\lO_0(\mathscr A)$ of $\mathscr A$ is of the form $\lO_n = \Z_\ell + \ell^n\lO_0$ for some non-negative integer $n$.

\begin{definition}
Let $\varphi: \mathscr A \rightarrow \mathscr B$ be an isogeny.
If $\lO_0(\mathscr A) \subset \lO_0(\mathscr B)$, we say that $\varphi$ is an \emph{RM-ascending} isogeny, if $\lO_0(\mathscr B) \subset \lO_0(\mathscr A)$ we say it is \emph{RM-descending}, otherwise $\lO_0(\mathscr A) = \lO_0(\mathscr B)$ and it is \emph{RM-horizontal}.
\end{definition}

\subsection*{Proof of Theorem~\ref{RMupIso}}
Theorem ~\ref{RMupIso} follows from Proposition~\ref{prop:neighborsLevelsRM} together with Remark~\ref{rem:correspEllEllAndEllEll}, and the observation that the $\lO_{n-1}$-lattice $\ell \lO_{n-1}\Lambda$ has order $\lO_{n-1}\cdot\lO(\Lambda)$.
\qed\\

In the following, we show that some structure of the graphs of horizontal isogenies at any level can be inferred from the structure at the maximal level: indeed, there is a graph homomorphism from any non-maximal level to the level above.

\begin{definition}[RM-predecessor]
Suppose $\lO_0(\mathscr A) = \lO_n$ with $n > 0$. Note that the kernel $\kappa$ of the unique RM-ascending isogeny of Proposition~\ref{RMupIso} is given by $(\lO_{n-1}T_\ell\mathscr A)/T_\ell\mathscr A$ (via Proposition~\ref{prop:correspondence}) and does not depend on the polarization.
The \emph{RM-predecessor} of $\mathscr A$ is the variety $\pr(\mathscr A) = \mathscr A / \kappa$, and we denote by $\mathrm{up}_\mathscr A : \mathscr A \ra \mathscr A / \kappa$ the canonical projection.
If $\xi$ is a principal polarization on $\mathscr A$, let $\pr(\xi)$ be the unique principal polarization induced by $\xi$ via $\mathrm{up}_\mathscr A$.
\end{definition}

\begin{proposition}\label{prop:liftingIsogeniesUp}
Suppose $n > 0$. For any principal polarization $\xi$ on $\mathscr A$, and any RM-horizontal
$\ll$-isogeny $\varphi : \mathscr A \rightarrow \mathscr B$ with respect to $\xi$, there is an $\ll$-isogeny $\tilde{\varphi} : \pr(\mathscr A) \rightarrow \pr(\mathscr B)$ with respect to $\pr(\xi)$ such that the following diagram commutes:
\begin{equation*}
\xymatrix{
\pr(\mathscr A) \ar[r]^{\tilde\varphi} & \pr(\mathscr B)  \\
\mathscr A \ar[u]^{\mathrm{up}_\mathscr A}  \ar[r]^{\varphi} & \mathscr B \ar[u]_{\mathrm{up}_\mathscr B}.
} 
\label{eq:ppav}
\end{equation*}
\end{proposition}
\begin{proof}
This follows from the fact that if $\Lambda$ is an $\lO_n$-lattice, and $\Gamma\in\mathscr L(\Lambda)$ is an $\ll$-neighbor of $\Lambda$, then $\ell\lO_{n-1}\Gamma \in \mathscr L(\ell\lO_{n-1}\Lambda)$.
\end{proof}

\section{$\ll$-isogenies preserving the real multiplication}

\subsection{$\ll$-neighbors and $\mathfrak l$-neighbors}
Let $\mathscr L_0(\Lambda)$ be the set of $(\ell,\ell)$-neighbors of the lattice $\Lambda$ with maximal real multiplication. These neighbors will be analysed through $\mathfrak l$-neighbors, for $\mathfrak l$ a prime ideal in $\lO_0$.
This will allow us to account for the possible splitting behaviors of $\ell$.
The relation between the set $\mathscr L_0(\Lambda)$ and the sets $\mathscr L_\mathfrak l(\Lambda)$ is given by the following proposition proved case-by-case in the following three sections, as Propositions~\ref{prop:L0LambdaInert}, \ref{prop:L0LambdaSplit} and \ref{prop:L0LambdaRamif}: 

\begin{proposition}\label{prop:classificationRMPreservingNeighbors}
Let $\Lambda$ be a lattice with maximal real multiplication. The set of $(\ell,\ell)$-neighbors with maximal real multiplication is
\[
\mathscr L_0(\Lambda) =  \left\{ \begin{array}{ll}
        \mathscr L_{\ell\lO_0}(\Lambda) & \mbox{if $\ell$ is inert in $K_0$},\\
        \mathscr L_{\mathfrak l_1}[\mathscr L_{\mathfrak l_2}(\Lambda)] = \mathscr L_{\mathfrak l_2}[\mathscr L_{\mathfrak l_1}(\Lambda)] & \mbox{if $\ell$ splits as $\mathfrak l_1\mathfrak l_2$ in $K_0$},\\
        \mathscr L_{\mathfrak l}[\mathscr L_{\mathfrak l}(\Lambda)] & \mbox{if $\ell$ ramifies as $\mathfrak l^2$ in $K_0$}.\end{array} \right.
\]
\end{proposition}

\subsubsection{The inert case} Suppose that $\ell$ is inert in $K_0$. Then, $\ell\cO_{K_0}$ is the unique prime ideal of $K_0$ above $\ell$. The orders in $K_\ell$ with maximal real multiplication are exactly the orders $\lO_{\ell^n\lO_0} = \lO_0 + \ell^n \lO_{K}$.
\begin{proposition}\label{prop:L0LambdaInert}
Let $\Lambda$ be a lattice with maximal real multiplication. If $\ell$ is inert in $K_0$, the set of $(\ell,\ell)$-neighbors with maximal real multiplication is $$\mathscr L_0(\Lambda) = \mathscr L_{\ell\lO_0}(\Lambda).$$
\end{proposition}
\begin{proof}
Since $\lO_0/\ell\lO_0 \isom \F_{\ell^2}$, $\Lambda/\ell\Lambda$ is a free $\lO(\Lambda)/\ell \lO(\Lambda)$-module of rank 1. In particular, it is a vector space over $\F_{\ell^2}$ of dimension 2, and thereby the $\lO_{0}$-stable maximal isotropic subspaces of $\Lambda/\ell\Lambda$ are $\F_{\ell^2}$-lines. Since any $\F_{\ell^2}$-line is isotropic,
$\mathscr L_{\ell\lO_0}(\Lambda)$ is precisely the set of $(\ell,\ell)$-neighbors preserving the maximal real multiplication.
\end{proof}

\begin{remark}
The structure of $\mathscr L_0(\Lambda)$ is then fully described by Proposition~\ref{prop:genericNeighbors}, with $\mathfrak l = \ell \lO_0$, and $N\mathfrak l = \ell^2$. In particular, $\mathscr L(\Lambda)$ consists
of $\ell^2 + 1$ neighbors with maximal real multiplication, and $\ell^3 + \ell$ with real multiplication by $\lO_1$.
\end{remark}

\subsubsection{The split case} Suppose that $\ell$ splits in $K_0$ as $\ell\cO_{K_0} = \mathfrak l_1 \mathfrak l_2$. 
The orders in $K_\ell$ with maximal real multiplication are exactly the orders $\lO_{\mathfrak f} = \lO_0 + \mathfrak f \lO_{K}$, where $\mathfrak f = \mathfrak l_1^m \mathfrak l_2^n$ for any non-negative integers $m$ and $n$.

\begin{lemma}\label{lemma:decompositionLambda}
Suppose $\Lambda$ has maximal real multiplication. Then, we have the orthogonal decomposition
$\Lambda/\ell\Lambda = (\mathfrak l_1\Lambda/\ell\Lambda) \perp (\mathfrak l_2\Lambda/\ell\Lambda).$
\end{lemma}

\begin{proof}
Let  $\lO = \lO(\Lambda)$. Since $\mathfrak l_1$ and $\mathfrak l_2$ are coprime and $\mathfrak l_1\mathfrak l_2 = \ell\lO_0$, the quotient $\lO/\ell\lO$ splits as $\mathfrak l_1\lO/\ell\lO \oplus \mathfrak l_2\lO/\ell\lO.$
It follows that $\Lambda/\ell\Lambda = (\mathfrak l_1\Lambda/\ell\Lambda) \oplus (\mathfrak l_2\Lambda/\ell\Lambda)$. Furthermore,
$\langle \mathfrak l_1\Lambda, \mathfrak l_2\Lambda\rangle = \langle \Lambda, \mathfrak l_1\mathfrak l_2\Lambda\rangle = \langle \Lambda, \ell\Lambda\rangle \subset \ell\Z_\ell,$
so $\mathfrak l_1\Lambda/\ell\Lambda \subset (\mathfrak l_2\Lambda/\ell\Lambda)^\perp$. The last inclusion is also an equality because both $\mathfrak l_1\Lambda/\ell\Lambda$ and $\mathfrak l_2\Lambda/\ell\Lambda$ have dimension 2.
\end{proof}

\begin{lemma}\label{lemma:caracSplitIso}
Suppose $\Lambda$ has maximal real multiplication. An $(\ell,\ell)$-neighbor $\Gamma \in \mathscr L(\Lambda)$ has maximal real multiplication if and only if there exist $\Gamma_1 \in \mathscr L_{\mathfrak l_1}(\Lambda)$ and $\Gamma_2 \in \mathscr L_{\mathfrak l_2}(\Lambda)$ such that $\Gamma = \mathfrak l_2\Gamma_1 + \mathfrak l_1\Gamma_2$.
\end{lemma}

\begin{proof}
First, let $\Gamma \in \mathscr L(\Lambda)$ be an $(\ell,\ell)$-neighbor with maximal real multiplication. Defining $\Gamma_i = \Gamma + \mathfrak l_i\Lambda$, we then have
$$\mathfrak l_2 \Gamma_1 + \mathfrak l_1 \Gamma_2 = (\mathfrak l_1 + \mathfrak l_2)\Gamma + \ell\Lambda = \lO_0\Gamma + \ell\Lambda = \Gamma.$$
By contradiction, suppose $\Gamma_i \not \in \mathscr L_i(\Lambda)$. Then, $\Gamma_i$ is either $\Lambda$ or $\mathfrak l_i\Lambda$. Suppose first that $\Gamma_i = \Lambda$. Then $\Gamma \subset \mathfrak l_i\Lambda$, and even $\Gamma = \mathfrak l_i\Lambda$ since $[\Lambda : \Gamma] = [\Lambda : \mathfrak l_i\Lambda] = \ell^2$. But the orthogonal decomposition of Lemma~\ref{lemma:decompositionLambda} implies that $\mathfrak l_i\Lambda/\Lambda$ is not isotropic, contradicting the fact that $\Gamma \in \mathscr L(\Lambda)$.

For the converse, suppose $\Gamma = \mathfrak l_2\Gamma_1 + \mathfrak l_1\Gamma_2$ for some $\Gamma_1 \in \mathscr L_{\mathfrak l_1}(\Lambda)$ and $\Gamma_2 \in \mathscr L_{\mathfrak l_2}(\Lambda)$. Then $\Gamma/\ell\Lambda$ is of dimension 2, so it suffices to prove that it is isotropic. Each summand $\mathfrak l_i\Gamma_j$ is isotropic, because it is of dimension 1, and Lemma~\ref{lemma:decompositionLambda} implies that $\mathfrak l_2\Gamma_1$ and $\mathfrak l_1\Gamma_2$ are orthogonal, so their sum $\Gamma$ is isotropic.
\end{proof}

\begin{proposition}\label{prop:L0LambdaSplit}
Suppose $\Lambda$ has maximal real multiplication. If $\ell$ splits in $K_0$ as $\ell\lO_0 = \mathfrak l_1 \mathfrak l_2$, the set of $(\ell,\ell)$-neighbors of $\Lambda$ with maximal real multiplication is $$\mathscr L_0(\Lambda) = \mathscr L_{\mathfrak l_1}[\mathscr L_{\mathfrak l_2}(\Lambda)] = \mathscr L_{\mathfrak l_2}[\mathscr L_{\mathfrak l_1}(\Lambda)].$$
\end{proposition}

\begin{proof}
For any $\Gamma_1 \in \mathscr L_{\mathfrak l_1}(\Lambda)$ and $\Gamma_2 \in \mathscr L_{\mathfrak l_2}(\Lambda)$, we have that $\mathfrak l_2\Gamma_1 + \mathfrak l_1\Gamma_2 \in  \mathscr L_{\mathfrak l_2}(\Gamma_1)$ and $\mathfrak l_2\Gamma_1 + \mathfrak l_1\Gamma_2 \in  \mathscr L_{\mathfrak l_1}(\Gamma_2)$. This proposition is thus a consequence of Lemma~\ref{lemma:caracSplitIso}.
\end{proof}

\begin{remark}
When $\ell$ splits in $K_0$, $\mathscr L_0(\Lambda)$ is then of size $\ell^2 + 2\ell + 1$, and the $\ell^3 - \ell$ other $(\ell,\ell)$-neighbors have real order $\lO_1$.
\end{remark}

\subsubsection{The ramified case} Suppose that $\ell$ ramifies in $K_0$ as $\ell\cO_{K_0} = \mathfrak l^2$. Then, $\lO_0/\ell\lO_0$ is isomorphic to $\F_\ell[\epsilon]$ with $\epsilon^2 = 0$. The orders in $K_\ell$ with maximal real multiplication are exactly the orders $\lO_{\mathfrak l^n} = \lO_0 + \mathfrak l^n \lO_{K}$.

\begin{proposition}\label{prop:L0LambdaRamif}
Suppose $\Lambda$ has maximal real multiplication. If $\ell$ splits in $K_0$ as $\ell\lO_0 = \mathfrak l^2$, the set of $(\ell,\ell)$-neighbors of $\Lambda$ with maximal real multiplication is $$\mathscr L_0(\Lambda) = \mathscr L_{\mathfrak l}[\mathscr L_{\mathfrak l}(\Lambda)].$$
\end{proposition}

\begin{proof}
Let $\Gamma \in \mathscr L_0(\Lambda)$. First, if $\Gamma = \mathfrak l \Lambda$, observe that for any $\Pi \in \mathscr L_\mathfrak l(\Lambda)$, we have $\mathfrak l \Lambda \in \mathscr L_\mathfrak l(\Lambda)$, and therefore $\Gamma \in \mathscr L_{\mathfrak l}[\mathscr L_{\mathfrak l}(\Lambda)]$. We can now safely suppose $\Gamma \neq \mathfrak l \Lambda$. Let $\Pi = \Gamma + \mathfrak l \Lambda$. We have the sequence of inclusions
$$\ell\Lambda \subset \mathfrak l \Pi \subset \Gamma \subsetneq \Pi \subset \Lambda.$$
By contradiction, suppose $\Pi = \Lambda$. Then, $\Gamma \cap \mathfrak l \Lambda = \ell\Lambda$. Since $\mathfrak l \Gamma \subset \Gamma \cap \mathfrak l \Lambda = \ell\Lambda$, it follows that
$\mathfrak l \Lambda = \mathfrak l \Pi = \mathfrak l \Gamma + \ell \Lambda \subset \ell \Lambda,$
a contradiction. Therefore $\Gamma \subsetneq \Pi \subsetneq \Lambda$, and each inclusion must be of index $\ell$. Then, $\Gamma \in \mathscr L_\mathfrak l (\Pi) \subset \mathscr L_\mathfrak l [\mathscr L_\mathfrak l (\Lambda)]$.

Let us now prove that $\mathscr L_{\mathfrak l}[\mathscr L_{\mathfrak l}(\Lambda)] \subset \mathscr L_0(\Lambda)$. Let $\Pi \in \mathscr L_{\mathfrak l}(\Lambda)$ and $\Gamma \in \mathscr L_{\mathfrak l}(\Pi)$. We have the sequence of inclusions
$$\ell\Lambda = \mathfrak l (\mathfrak l \Lambda) \subset_\ell \mathfrak l \Pi \subset_\ell \Gamma \subset_\ell \Pi \subset_\ell \Lambda,$$
where $\subset_\ell$ means that the first lattice is of index $\ell$ in the second. Therefore $\ell\Lambda \subset \Gamma \subset \Lambda$, and $\Gamma / \ell \Lambda$ is of dimension 2 over $\F_\ell$.
Since $\Gamma / \mathfrak l \Lambda$ is a line, there is an element $\pi \in \Pi$ such that $\Pi = \Z_\ell \pi + \mathfrak l \Lambda$. Similarly, $\Pi / \mathfrak l \Gamma$ is a line, so there is an element
$\gamma \in \Gamma$ such that $\Gamma = \Z_\ell\gamma + \mathfrak l \pi + \ell \Lambda$. Therefore, writing $x = \gamma + \ell\Lambda$ and $y = \pi + \ell\Lambda$, $\Gamma / \ell \Lambda$ is generated as an $\F_\ell$-vector space by $x$ and
$\epsilon y$. Since $\gamma \in \Gamma \subset \Pi =  \Z_\ell \pi + \mathfrak l \Lambda$, there exist $a\in\Z_\ell$ and $z \in \Lambda/\ell\Lambda$ such that $x = ay + \epsilon z$. Then,
$$\langle x, \epsilon y \rangle_\ell = \langle a y, \epsilon y \rangle_\ell + \langle \epsilon z, \epsilon y \rangle_\ell = a \langle y, \epsilon y \rangle_\ell + \langle z, \epsilon^2 y \rangle_\ell = 0,$$
where the last equality uses Lemma~\ref{lemma:realOrbitIsotropic}, and the fact that $\epsilon^2 = 0$. So $\Gamma/\ell\Lambda$ is maximal isotropic, and $\Gamma \in \mathscr L(\Lambda)$. Furthermore $\epsilon x = a\epsilon y$, and $\epsilon y = 0$ are both in $\Gamma/\ell\Lambda$, so the latter is $\F_\ell[\epsilon]$-stable, so $\Gamma$ is $\lO_0$-stable. This proves that $\Gamma \in \mathscr L_0(\Lambda)$.
\end{proof}

\begin{remark}\label{rem:numberofEllEllinRamifCase}
We can deduce from Lemma~\ref{lemma:subspacesEpsilon} that $|\mathscr L_0(\Lambda)| = \ell^2 + \ell + 1$. In fact, for any two distinct lattices $\Pi_1, \Pi_2 \in \mathscr L_{\mathfrak l}(\Lambda)$, we have $\mathscr L_{\mathfrak l}(\Pi_1) \cap \mathscr L_{\mathfrak l}(\Pi_2) = \{\mathfrak l \Lambda\}$.
\end{remark}

\subsection{Locally maximal real multiplication and $\ll$-isogenies}\label{subsec:ellellMaxRM}

Fix again a principally polarizable absolutely simple ordinary abelian surface $\mathscr A$ over $\F_q$, with endomorphism algebra $K$, and $K_0$ the maximal real subfield of $K$. 
Now suppose that $\mathscr A$ has locally maximal real multiplication at $\ell$. Recall from Theorem~\ref{thm:classificationOrdersMaxRM} that any such locally maximal real order is of the form $\lO_\mathfrak f = \lO_0 + \mathfrak f \lO_K$, for some $\lO_0$-ideal $\mathfrak f$.
The structure of $\mathfrak l$-isogeny graphs as described by Theorem~\ref{thm:lisogenyvolcanoes} can be used to describe graphs of $\ll$-isogenies preserving the real multiplication, via Theorem~\ref{thm:ellellLCombinations}.

\subsection*{Proof of Theorem~\ref{thm:ellellLCombinations}}
This theorem is a direct consequence of Proposition~\ref{prop:classificationRMPreservingNeighbors} translated to the world of isogenies via Remark~\ref{rem:correspEllEllAndEllEll}.
\qed

\begin{remark}\label{rem:RMPreservingEllEllDoNotDependOnPol}
Note that in particular, Theorem~\ref{thm:ellellLCombinations} implies that the kernels of the $\ll$-isogenies $\mathscr A \ra \mathscr B$ preserving the real multiplication do not depend on the choice of a polarization $\xi$ on $\mathscr A$.
\end{remark}

\subsubsection{The inert and ramified cases}
Combining Theorem~\ref{thm:lisogenyvolcanoes} and Theorem~\ref{thm:ellellLCombinations} allows us to describe the graph of $\ll$-isogenies with maximal local real multiplication at $\ell$.
 For purpose of exposition, we assume from now on that the primitive quartic CM-field $K$ is different from $\Q(\zeta_5)$, but the structure for $\Q(\zeta_5)$ can be deduced in the same way (bearing in mind that in that case, $\cO_{K_0}^\times$ is of index 5 in $\cO_{K}^\times$).
 Let $\mathscr A$ be any principally polarizable abelian variety with order $\cO$, with maximal real multiplication locally at $\ell$.
When $\ell$ is inert in $K_0$, the connected component of $\mathscr A$ in the $\ll$-isogeny graph (again, for maximal local real multiplication) is exactly the volcano $\mathscr V_\mathfrak l(\cO)$ (see Notation~\ref{notation:lVolcano}).
When $\ell$ ramifies as $\mathfrak l^2$ in $K_0$, 
the connected component of $\mathscr A$ in the graph of $\mathfrak l$-isogenies is isomorphic to the graph $\mathscr V_\mathfrak l(\cO)$, and the graph of $\ll$-isogenies
can be constructed from it as follows: on the same set of vertices, add an edge in the $\ll$-graph between $\mathscr B$ and $\mathscr C$ for each path of length 2 between $\mathscr B$ and $\mathscr C$ in the $\mathfrak l$-volcano; each vertex $\mathscr B$ has now $\ell^2 + 2\ell + 1$ outgoing edges, while there are only $\ell^2 + \ell + 1$ possible kernels of RM-preserving $\ll$-isogenies (see Remark~\ref{rem:numberofEllEllinRamifCase}). This is because the edge corresponding to the canonical projection $\mathscr B \rightarrow \mathscr B / \mathscr B[\mathfrak l]$ has been accounted for $\ell + 1$ times. Remove $\ell$ of these copies, and the result is exactly the graph of $\ll$-isogenies.

\begin{example}
Suppose $\ell = 2$ ramifies in $K_0$ as $\mathfrak l^2$, and $\mathfrak l$ is principal in $\cO_{K_0}$. Suppose further that $\mathfrak l$ splits in $K$ into two prime ideals of order $4$ in $\Cl(\cO_K)$.
Then, the first four levels of any connected component of the $\ll$-isogeny graph for which the largest order is $\cO_K$ are isomorphic to the graph of Figure~\ref{fig:exampleOfEllEllRamif}. The underlying $\mathfrak l$-isogeny volcano is represented with dotted nodes and edges. Since $\mathfrak l$ is principal in $\cO_{K_0}$, it is an undirected graph, and we represent it as such. The level 0, i.e., the surface of the volcano, is the dotted cycle of length 4 at the center. The circles have order $\cO_{K}$, the squares have order $\cO_{K_0} + \mathfrak l \cO_{K}$, the diamonds $\cO_{K_0} + \ell \cO_{K}$, and the triangles $\cO_{K_0} + \mathfrak l^3 \cO_{K}$.
\end{example}

\begin{figure}
\includegraphics{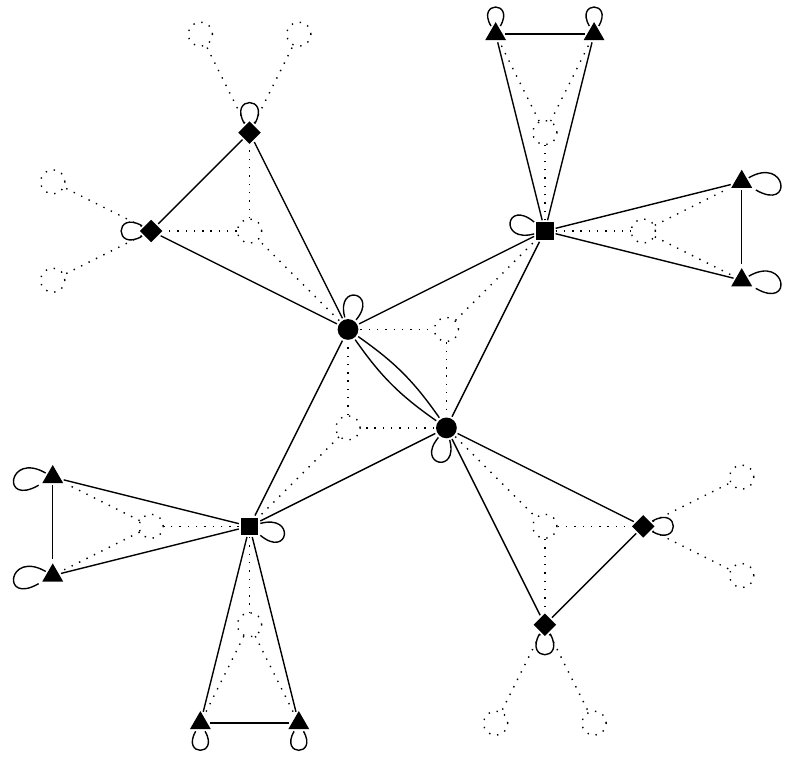}
\caption{\label{fig:exampleOfEllEllRamif} An example of $\ll$-isogeny graph, when $\ell$ ramifies in $K_0$.}
\end{figure}

\subsubsection{The split case}
For simplicity, suppose again that the primitive quartic CM-field $K$ is a different from $\Q(\zeta_5)$.
Let $\mathscr A$ be any principally polarizable abelian variety with order $\cO$, with maximal real multiplication locally at $\ell$.
The situation when $\ell$ splits as $\mathfrak l_1\mathfrak l_2$ in $K_0$ (with $\mathfrak l_1$ and $\mathfrak l_2$ principal in $\cO \cap K_0$) is a bit more delicate because the $\mathfrak l_1$ and $\mathfrak l_2$-isogeny graphs need to be carefully pasted together.
Let $\mathscr G_{\mathfrak l_1,\mathfrak l_2}(\mathscr A)$ be the connected component of $\mathscr A$ in the labelled isogeny graphs whose edges are $\mathfrak l_1$-isogenies (labelled $\mathfrak l_1$) and $\mathfrak l_2$-isogenies (labelled~$\mathfrak l_2$).
The graph of $\ll$-isogenies is the graph on the same set of vertices, such that the number of edges between two vertices $\mathscr B$ and $\mathscr C$ is exactly the number of paths of length 2 from $\mathscr B$ to $\mathscr C$, whose first edge is labelled $\mathfrak l_1$ and second edge is labelled $\mathfrak l_2$. It remains to fully understand the structure of the graph $\mathscr G_{\mathfrak l_1,\mathfrak l_2}(\mathscr A)$.
Like for the cases where $\ell$ is inert or ramified in $K_0$, we would like a complete characterization of the structure of the isogeny graph, i.e., a description that is sufficient to construct an explicit model of the abstract graph.

Without loss of generality, suppose $\cO$ is locally maximal at $\ell$. Then, the endomorphism ring of any variety in $\mathscr G_{\mathfrak l_1,\mathfrak l_2}(\mathscr A)$ is characterized by the conductor $\mathfrak l_1^m\mathfrak l_2^n$ at $\ell$, and we denote by $\cO_{m,n}$ the corresponding order. The graph $\mathscr G_{\mathfrak l_1,\mathfrak l_2}(\mathscr A)$ only depends on the order, so we also denote it $\mathscr G_{\mathfrak l_1,\mathfrak l_2}(\cO)$. For simplicity of exposition, let us assume that $\mathfrak l_1$ and $\mathfrak l_2$ are principal in $\cO\cap K_0$, so that the $\mathfrak l_i$-isogeny graphs are volcanoes.

\begin{definition}[cyclic homomorphism]
Let $\mathscr X$ and $\mathscr Y$ be two graphs. A graph homomorphism $\psi : \mathscr X \rightarrow \mathscr Y$ is a \emph{cyclic homomorphism} if each edge of $\mathscr X$ and $\mathscr Y$ can be directed in such a way that $\psi$ becomes a homomorphism of directed graphs, and each undirected cycle in $\mathscr X$ becomes a directed cycle.
\end{definition}

\begin{lemma}\label{lemma:easyLemmaConnectedComp}
Let $\mathscr X$, $\mathscr Y$ and $\mathscr Y'$ be connected, $d$-regular graphs, with $d \leq 2$, such that $\mathscr Y$ and $\mathscr Y'$ are isomorphic. If $\varphi : \mathscr Y \rightarrow \mathscr X$ and $\varphi' : \mathscr Y' \rightarrow \mathscr X$ are two cyclic homomorphisms, there is an isomorphism $\psi : \mathscr Y \rightarrow \mathscr Y'$ such that $\varphi = \varphi' \circ \psi$.
\end{lemma}

\begin{proof}
The statement is trivial if $d$ is 0 or 1. Suppose $d = 2$, i.e., $\mathscr X$, $\mathscr Y$ and $\mathscr Y'$ are cycles.
Let $\mathscr X$ be
the cycle $x_0 - x_1 - \dots - x_m$, with $x_m = x_0$. Similarly, $\mathscr Y$ is the cycles $y_0 - y_1 - \dots - y_n$, with $y_n = y_0$. Without loss of generality, $\varphi(y_0) = x_0$ and $\varphi(y_1) = x_1$.
There is a direction on the edges of $\mathscr X$ and $\mathscr Y$ such that $\varphi$ becomes a homomorphism of directed graphs, and $\mathscr Y$ becomes a directed cycle.
Without loss of generality, the direction of $\mathscr Y$ is given by $y_i \rightarrow y_{i+1}$.
Since $y_0 \rightarrow y_{1}$, we have $\varphi(y_0) \rightarrow \varphi(y_{1})$, hence $x_0 \rightarrow x_{1}$. Since $y_1 \rightarrow y_{2}$, we must also have $x_1 \rightarrow \varphi(y_{2})$, so $\varphi(y_{2}) \neq x_0$ and therefore $\varphi(y_{2}) = x_2$, and as a consequence $x_1 \rightarrow x_{2}$.
Repeating inductively, we obtain $x_i \rightarrow x_{i+1}$ for all $i \leq m$, and $\varphi(y_{i}) = x_{i \text{ mod } m}$ for all $i \leq n$.

Similarly, any direction on $\mathscr X$ and $\mathscr Y'$ such that $\mathscr Y'$ is a directed cycle and $\varphi'$ becomes a homomorphism of directed graphs turns $\mathscr X$ into a directed cycle. Without loss of generality, it is exactly the directed cycle $x_0 \rightarrow x_1 \rightarrow \dots \rightarrow x_m$ (if it is the other direction, simply invert the directions of $\mathscr Y'$). There is then an enumeration $\{y_i'\}_{i=0}^n$ of $\mathscr Y'$ such that $\varphi'(y_i') = x_i$, and $y_i' \rightarrow y'_{i+1}$ for each $i$. The isomorphism $\psi$ is then simply given by $\psi(y_i) = y_i'$.
\end{proof}

\begin{proposition}\label{prop:structureVolcamicMess}
The graph $\mathscr G_{\mathfrak l_1,\mathfrak l_2}(\cO)$, with edges labelled by $\mathfrak l_1$ and $\mathfrak l_2$, and bi-levelled by $(v_{\mathfrak l_1}, v_{\mathfrak l_2})$, is isomorphic to the unique (up to isomorphism) graph $\mathscr G$ with edges labelled by $\mathfrak l_1$ and $\mathfrak l_2$, and bi-levelled by a pair $(v_1, v_2)$, satisfying:
\begin{enumerate}[label=(\roman*)]
\item\label{property:l1l2volcanoes} For $i= 1,2$, the subgraph of $\mathscr G$ containing only the edges labelled by $\mathfrak l_i$ is a disjoint union of $\ell$-volcanoes, levelled by $v_{i}$,
\item\label{property:l1l2compatibility} For $i \neq j$, if $u$ and $v$ are connected by an $\mathfrak l_i$-edge, then $v_{j}(u) = v_j(v)$,
\item\label{property:l1l2CayleySubgraphs} For any non-negative integers $m$, and $n$, let $\mathscr G_{m,n}$ be the subgraph containing the vertices $v$ such that $(v_1(v),v_2(v)) = (m,n)$. Then,
\begin{enumerate}[label=(\roman*)]
\item $\mathscr G_{0,0}$ is isomorphic to the Cayley graph $\mathscr C_{0,0}$ of the subgroup of $\Pic(\cO)$ with generators the invertible ideals of the order $\cO$ above $\ell$, naturally labelled by $\mathfrak l_1$ and $\mathfrak l_2$,
\item each connected component of $\mathscr G_{m,n}$ is isomorphic to the Cayley graph $\mathscr C_{m,n}$ of the subgroup of $\Pic(\cO_{m,n})$ with generators the invertible ideals of the order $\cO_{m,n}$ above $\ell$, naturally labelled by $\mathfrak l_1$ and $\mathfrak l_2$,
\end{enumerate}
\item\label{property:l1l2commute} For any two vertices $u$ and $v$ in $\mathscr G$, there is a path of the form $u -_{\mathfrak l_1} w -_{\mathfrak l_2} v$ if and only if there is a path of the form $u -_{\mathfrak l_2} w' -_{\mathfrak l_1} v$ (where $-_{\mathfrak l_i}$ denotes an edge labelled by $\mathfrak l_i$).
\end{enumerate}
\end{proposition}

\begin{proof}
First, it is not hard to see that $\mathscr G_{\mathfrak l_1,\mathfrak l_2}(\cO)$ satisfies all these properties. Properties~\ref{property:l1l2volcanoes} and~\ref{property:l1l2compatibility} follow from Proposition~\ref{prop:frakLStructure} and Theorem~\ref{thm:lisogenyvolcanoes}. Property~\ref{property:l1l2CayleySubgraphs} follows from the free CM-action of $\Pic(\cO_{m,n})$ on the corresponding isomorphism classes. Property~\ref{property:l1l2commute} follows from the fact that $\mathscr A[\mathfrak l_1] \oplus \mathscr A[\mathfrak l_2]$ is a direct sum.

Let $\mathscr G$ and $\mathscr G'$ be two graphs with these properties.
For $i = 1,2$, let $\pr_i$ (respectively, $\pr'_i$) be the predecessor map induced by the volcano structure of the $\mathfrak l_i$-edges on $\mathscr G$ (respectively, on $\mathscr G'$).
We will construct an isomorphism $\Psi : \mathscr G \rightarrow \mathscr G'$ by starting with the isomorphism between $\mathscr G_{0,0}$ and $\mathscr G'_{0,0}$ and extending it on the blocks $\mathscr G_{m,n}$ and $\mathscr G'_{m,n}$ one at a time.
Let $n>0$, and suppose, by induction, that $\Psi$ has been defined exactly on the blocks $\mathscr G_{i,j}$ for $i + j < n$. Let us extend $\Psi$ to the blocks $\mathscr G_{m,n-m}$ for $m = 0,\dots,n$ in order.

Both $\mathscr G_{0,n}$ and $\mathscr G'_{0,n}$ have the same number of vertices, and their connected components are all isomorphic $\mathscr C_{0,n}$, which are of degree $d$ at most 2.
We have the graph homomorphism $\pr_2 : \mathscr G_{0,n} \rightarrow \mathscr G_{0,n-1}$. Let $S$ be the set of connected components of $\mathscr G_{0,n}$. Define the equivalence relation on $S$
$$A \sim B \Longleftrightarrow \pr_2(A) = \pr_2(B).$$
Similarly define the equivalence relation $\sim'$ on the set $S'$ of connected components of $\mathscr G'_{0,n}$. 
Observe that each equivalence class for either $\sim$ or $\sim'$ has same cardinality, so one can choose a bijection $\Theta : S \rightarrow S'$ such that for any $A \in S$, we have $\Psi(\pr_2(A)) = \pr'_2(\Theta(A))$. It is not hard to check that $\pr_2$ and $\pr'_2$ are cyclic homomorphisms, using Property~\ref{property:l1l2commute}. From Lemma~\ref{lemma:easyLemmaConnectedComp}, for each $A \in S$, there is a graph isomorphism $\psi_A : A \rightarrow \Theta(A)$ such that for any $x \in A$, $\pr'_2( \psi_A(x)) = \Psi( \pr_2(x))$.
Let $\hat\Psi$ be the map extending $\Psi$ by sending any $x \in \mathscr G_{0,n}$ to $\psi_A(x)$, where $A$ is the connected component of $x$ in $\mathscr G_{0,n}$.
We need to show that it is a graph isomorphism. Write $\mathscr D$ and $\mathscr D'$ the domain and codomain of $\Psi$. The map 
$\hat\Psi$, restricted and corestricted to $\mathscr D$ and $\mathscr D'$ is exactly $\Psi$ so is an isomorphism. Also, the restriction and corestriction 
to $\mathscr G_{0,n}$ and $\mathscr G'_{0,n}$ is an isomorphism, by construction. Only the edges between $\mathscr G_{0,n}$ and $\mathscr D$ (respectively $\mathscr G'_{0,n}$ and $\mathscr D'$) might cause trouble. The only edges between $\mathscr G_{0,n}$ and $\mathscr D$ are actually between $\mathscr G_{0,n}$ and $\mathscr G_{0,n-1}$, and are of the form $(x,\pr_2(x))$. But $\Psi$ was precisely constructed so that $\Psi(\pr_2(x)) = \pr'_2(\Psi(x))$, so $\hat \Psi$ is indeed an isomorphism.

Now, let $0<m<n$ and suppose that $\Psi$ has been extended to the components $\mathscr G_{i,n-i}$ for each $i < m$. Let us extend it to $\mathscr G_{m,n-m}$. Since $m>0$ and $n-m >0$, the graph $\mathscr C_{m,n-m}$ is a single point, with no edge.
Let us now prove that for any pair $(x_1,x_2)$, where $x_1$ is a vertex in $\mathscr G_{m-1,n-m}$ and $x_2$ in $\mathscr G_{m,n-m-1}$ such that $\pr_2(x_1) = \pr_1(x_2)$, there is a unique vertex $x$ in $\mathscr G_{m,n-m}$ such that $(x_1,x_2) = (\pr_1(x), \pr_2(x))$.
First, we show that for any vertex $x \in \mathscr G_{m,n-m}$, we have
$$\pr_1^{-1}(\pr_1(x)) \cap \pr_2^{-1}(\pr_2(x)) = \{x\}.$$
Let $z = \pr_1(\pr_2(x))$. Let $X = \pr_{1}^{-1}(\pr_1(x))$ and $Y = \pr_{1}^{-1}(z)$. 
From Property~\ref{property:l1l2compatibility}, $z$ and $\pr_1(x)$ are at the same $v_1$-level, so
from Property~\ref{property:l1l2volcanoes}, we have $|X| = |Y|$. For any $y \in Y$, we have $\pr_1(x) -_{\mathfrak l_2} z -_{\mathfrak l_1} y$, so there is a vertex $x'$ such that $\pr_1(x) -_{\mathfrak l_1} x' -_{\mathfrak l_2} y$. Then, $v_1(x') = v_1(y) = v_1(\pr_1(x)) - 1$, and therefore $x' \in X$. This implies that $\pr_2$ induces a surjection $\tilde{\pr}_2 : X \rightarrow Y$, which is a bijection since $|X| = |Y|$. So
$$X \cap \pr_2^{-1}(\pr_2(x)) = X \cap \tilde{\pr}_2^{-1}(\pr_2(x)) = \{x\}.$$
Now, an elementary counting argument shows that $x \mapsto (\pr_1(x), \pr_2(x))$ is a bijection between the vertices of $\mathscr G_{m,n-m}$ and the pairs $(x_1,x_2)$, where $x_1$ is a vertex in $\mathscr G_{m-1,n-m}$ and $x_2$ in $\mathscr G_{m,n-m-1}$ such that $\pr_2(x_1) = \pr_1(x_2)$. This property also holds in $\mathscr G'$, and we can thereby define $\psi: \mathscr G_{m,n-m} \rightarrow \mathscr G'_{m,n-m}$ as the bijection sending any vertex $x$ in $\mathscr G_{m,n-m}$ to the unique vertex $x'$ in $\mathscr G'_{m,n-m}$ such that $$(\pr'_1(x'), \pr'_2(x')) = (\Psi(\pr_1(x)), \Psi(\pr_2(x))).$$
It is then easy to check that the extension of $\Psi$ induced by $\psi$ is an isomorphism.
The final step, extending on $\mathscr G_{n,0}$, is similar to the case of $\mathscr G_{0,n}$. This concludes the induction, and proves that $\mathscr G$ and $\mathscr G'$ are isomorphic.
\end{proof}

\section{Applications to ``going up" algorithms}\label{sec:goingUp}
\subsection{Largest reachable orders}
The results from Section~\ref{subsec:ellellLevelsRM} and Section~\ref{subsec:ellellMaxRM} on the structure of the graph of $\ll$-isogenies allow us to determine exactly when there exists a sequence of $\ll$-isogenies leading to a surface with maximal local order at $\ell$. When there is no such path, one can still determine the largest reachable orders.

\begin{proposition}\label{prop:ellellIncreasingLocalOrder}
Suppose $\mathscr A$ has maximal local real order, and $\lO(\mathscr A) = \lO_\mathfrak f$.
\begin{enumerate}[label=(\roman*)]
\item If $\ell$ divides $\mathfrak f$, there is a unique $\ll$-isogeny to a surface with order $\lO_{\ell^{-1}\mathfrak f}$.
\item If $\ell$ ramifies in $K_0$ as $\mathfrak l^2$ and $\mathfrak f = \mathfrak l$, then there exists an $\ll$-isogeny to a surface with maximal local order if and only if $\mathfrak l$ is not inert in $K$. It is unique if $\mathfrak l$ is ramified, and there are two if it splits.
\item If $\ell$ splits in $K_0$ as $\mathfrak l_1\mathfrak l_2$, and $\mathfrak f = \mathfrak l_1^i$ for some $i > 0$, then there exists an $\ll$-isogeny to a surface with local order $\lO_{\mathfrak l_1^{i-1}}$ if and only if $\mathfrak l_2$ is not inert in $K$. It is unique if $\mathfrak l_2$ is ramified, and there are two if it splits. Also, there always exist an $\ll$-isogeny to a surface with local order $\lO_{\mathfrak l_{1}^{i-1}\mathfrak l_2}$.
\end{enumerate}
\end{proposition}

\begin{proof}
This is a straightforward case-by-case analysis of Propositions~\ref{prop:frakLStructure} and Theorem~\ref{thm:ellellLCombinations}.
\end{proof}

\begin{definition}[parity of $\mathscr A$]
Suppose $\mathscr A$ has real order $\lO_n = \Z_\ell + \ell^n \lO_0$. Construct $\mathscr B$ by the RM-predecessor of  $\mathcal A$ $n$ times, i.e., 
$\mathscr B = \pr(\pr(\dots \pr(\mathscr A)\dots))$ is the (iterated) RM-predecessor of $\mathscr A$ that has maximal real local order. 
Let $\mathfrak f$ be the conductor of $\lO(\mathscr B)$. The \emph{parity} of $\mathscr A$ is 0 if $N(\mathfrak f\cap \lO_0)$ is a square, and 1 otherwise.
\end{definition}

\begin{remark}
The parity is always 0 if $\ell$ is inert in $K_0$.
\end{remark}

\begin{theorem}\label{thm:ellellsurface}
For any $\mathscr A$, there exists a sequence of $\ll$-isogenies starting from $\mathscr A$ and ending at a variety with maximal local order, except in the following two cases:
\begin{enumerate}[label=(\roman*)]
\item $\mathscr A$ has parity 1, $\ell$ splits in $K_0$ as $\mathfrak l_1\mathfrak l_2$, and both $\mathfrak l_1$ and $\mathfrak l_2$ are inert in $K$, in which case the largest reachable local orders are $\lO_0 + \mathfrak l_1\lO_K$ and $\lO_0 + \mathfrak l_2\lO_K$;
\item $\mathscr A$ has parity 1, $\ell$ ramifies in $K_0$ as $\mathfrak l^2$, and $\mathfrak l$ is inert in $K$, in which case the largest reachable local order is $\lO_0 + \mathfrak l\lO_K$.
\end{enumerate}
\end{theorem}

\begin{proof}
First, from Propositon~\ref{prop:liftingIsogeniesUp}, there is a sequence of $\ll$-isogenies starting from $\mathscr A$ and ending at a variety with maximal local order if and only if there is such a path that starts by a sequence of isogenies up to $\mathscr B = \pr(\pr(\dots \pr(\mathscr A)\dots))$, and then only consists of $\ll$-isogenies preserving the maximality of the local real order.  It is therefore sufficient to look at sequences of RM-preserving $\ll$-isogenies from $\mathscr B$, which has by construction the same parity $s$ as $\mathscr A$.

From Proposition~\ref{prop:ellellIncreasingLocalOrder}, there is a path from $\mathscr B$ to a surface $\mathscr C$ with local order $
\lO(\mathscr C) = \lO_{\mathfrak l^s}$ where $\mathfrak l$ is a prime ideal of $\lO_0$ above $\ell$, and $s$ is the parity of $\mathscr A$. We are done if the parity is 0.
Suppose the parity is 1. From Propositions~\ref{prop:frakLStructure} and Theorem~\ref{thm:ellellLCombinations}, one can see that there exists a sequence of RM-preserving $\ll$-isogeny from $\mathscr C$ which changes the parity to 0 if and only if $\ell$ ramifies in $K_0$ as $\mathfrak l^2$ and $\mathfrak l$ is not inert in $K$, or $\ell$ splits in $K_0$ as $\mathfrak l_1\mathfrak l_2$ and either $\mathfrak l_1$ or $\mathfrak l_2$ is not inert in $K$. This concludes the proof.
\end{proof}

\subsection{A ``going up" algorithm}

In many applications (in particular, the CM method in genus 2 based on the CRT) it is useful to find a chain of isogenies to a principally polarized abelian surface with maximal endomorphism ring starting from any curve whose Jacobian is in the given isogeny class.
Lauter and Robert \cite[\S 5]{lauter-robert} propose a probabilistic algorithm to construct a principally polarized abelian variety whose endomorphism ring is maximal. That algorithm is heuristic, and the probability of failure is difficult to analyze. We now apply our structural results from Subsection~\ref{subsec:ellellMaxRM} to some of their ideas to give a provable algorithm.

\subsubsection{Prior work of Lauter--Robert}
Given a prime $\ell$ for which we would like to find an isogenous abelian surface over $\F_q$ with maximal local endomorphism ring at $\ell$, suppose that $\alpha = \ell^e \alpha'$ for some $\alpha' \in \cO_K$ and some $e > 0$. 
To find a surface $\mathscr A'/\F_q$ for which $\alpha / \ell^e \in \End(\mathscr A')$, Lauter and Robert \cite[\S 5]{lauter-robert} use $(\ell, \ell)$-isogenies and a test for  
whether $\alpha / \ell^e \in \End(\mathscr A')$. In fact, $\alpha / \ell^e \in \End(\mathscr A')$ is equivalent to testing that $\alpha(\mathscr A'[\ell^e]) = 0$, i.e., $\alpha$ is trivial on the $\ell^e$-torsion of $\mathscr A$. 
To guarantee that, one defines an ``obstruction" $N_e = \# \alpha(\mathscr A[\ell^e])$ 
that measures the failure of $\alpha/\ell^e$ to be an endomorphism of $\mathscr A'$. To construct an abelian surface that contains the element $\alpha / \ell^e$ as endomorphism, one uses $(\ell, \ell)$-isogenies iteratively in order to decrease the associated obstruction $N_e$ (this is in essence the idea of \cite[Alg.21]{lauter-robert}).  

To reach an abelian surface with maximal local endomorphism ring at $\ell$, Lauter and Robert look at the structure of $\End(\mathscr A) \otimes_{\Z}\Z_\ell$ as a $\Z_\ell$-module and define an obstruction via a particular choice of a $\Z_\ell$-basis \cite[Alg.23]{lauter-robert}.

\subsubsection{Refined obstructions and provable algorithm}
Theorem~\ref{thm:ellellsurface} above gives a provable ``going up" algorithm that runs in three main steps: 1) it uses $(\ell, \ell)$-isogenies to reach a surface with maximal local real endomorphism ring at $\ell$; 2) it reaches the largest possible order via $(\ell, \ell)$-isogenies as in Theorem~\ref{thm:ellellsurface}; 3) if needed, it makes a last step to reach maximal local endomorphism ring via a cyclic isogeny. To implement 1) and 2), one uses refined obstructions, which we now describe in detail. 

\subsubsection{``Going up" to maximal real multiplication.}\label{subsec:surfacingToMaxRM}
Considering the local orders $\lO_0 = \cO_{K_0} \otimes_\Z \Z_\ell$ and $\Z_\ell[\pi + \pi^{\dagger}]$, choose a $\Z_\ell$-basis $\{1, \beta / \ell^e\}$ for $\lO_0$ such that $\beta \in \Z[\pi, \pi^{\dagger}]$ and we apply a ``real-multiplication" modification of \cite[Alg.21]{lauter-robert} to $\beta$.  
Thus, given an abelian surface $\mathscr A$ with endomorphism algebra isomorphic to $K$,  define the obstruction for $\mathscr A$ to have maximal real multiplication at $\ell$ as 
$$
N_0(\mathscr A) = e - \max \{\epsilon \colon \beta(\mathscr A[\ell^{\epsilon}]) = 0 \}. 
$$
Clearly, $\mathscr A$ will have maximal real endomorphism ring at $\ell$ if and only if $N_0(\mathscr A) = 0$. The following simple lemma characterizes the obstruction:

\begin{lemma}\label{lem:obstr-cond-real}
The obstruction $N_0(\mathscr A)$ is equal to the valuation at $\ell$ of the conductor of the real multiplication $\cO_0(\mathscr A) \subset \cO_{K_0}$.  
\end{lemma}

\begin{proof}
Using the definition of $N_0(\mathscr A)$ and the fact that $\beta / \ell^\epsilon \in \cO(\mathscr A)$ if and only if $\beta(\mathscr A[\ell^\epsilon]) = 0$, it follows that 
$$
\displaystyle \Z_\ell + \beta / \ell^{e - N_0(\mathscr A)} \Z_\ell \subseteq \lO_0(\mathscr A) \subsetneq \Z_\ell + \beta / \ell^{e - N_0(\mathscr A) + 1} \Z_\ell.   
$$
Since all orders of $\cO_{K_0}$ are of the form $\Z + c \cO_{K_0}$ for some $c \in \Z_{>0}$, by localization at $\ell$ one sees that
$$
\displaystyle \lO_0(\mathscr A) = \Z_\ell + \beta / \ell^{e - N_0(\mathscr A)} \Z_\ell = \Z_\ell + \ell^{N_0(\mathscr A)} \lO_0, 
$$ 
i.e., the valuation at $\ell$ of the conductor of $\cO_0(\mathscr A)$ is $N_0(\mathscr A)$. 
\end{proof}

The lemma proves the following algorithm works (i.e., that there always exists a neighbor decreasing the obstruction $N_0$):

\begin{algorithm}
\caption{Surfacing to maximal real endomorphism ring}
\begin{algorithmic}[1]
\REQUIRE An abelian surface $\mathscr A / \F_q$ with endomorphism algebra $K = \End(\mathscr A) \otimes \Q$, and a prime number $\ell$. 
\ENSURE An isogenous abelian surface $\mathscr A' / \F_q$ with $\lO_0(\mathscr A') = \lO_0$.  
\STATE $\beta \gets$ an element $\beta \in \Z[\pi, \overline{\pi}]$ such that $\{1, \beta / \ell^e\}$ is a $\Z_\ell$-basis for $\lO_0$.
\STATE Compute $N_0(\mathscr A) := e - \max\{\epsilon \colon \beta(\mathscr A[\ell^{\epsilon}]) = 0\}$
\IF{$N_0(\mathscr A) = 0$}\label{nplus}
	\RETURN $\mathscr A$
\ENDIF 
\STATE $\cL \leftarrow$ list of maximal isotropic $\kappa \subset \mathscr A[\ell]$ with $\kappa \cap \beta(\mathscr A[\ell^{e - N_0(\mathscr A) +1}]) \ne \emptyset$
\FOR{$\kappa \in \cL$}
	\STATE Compute $N_0(\mathscr A/\kappa) := e - \max\{\epsilon \colon \beta(\mathscr (A/\kappa) [\ell^{\epsilon}]) = 0\}$
	\IF{$N_0(\mathscr A/\kappa, \epsilon) < N_0(\mathscr A, \epsilon)$}
		\STATE $\mathscr A \leftarrow \mathscr A/\kappa$ and \textbf{go to} Step~\ref{nplus}
	\ENDIF
\ENDFOR
\end{algorithmic}
\label{alg:surf-real}
\end{algorithm}

\subsubsection{Almost maximal order with $\ll$-isogenies.} 
For each prime $\ell$, use the going-up algorithm (Algorithm~\ref{alg:surf-real}), until
$\cO_0(\mathscr A) = \cO_{K_0}$. Let $\ell$ be any prime and let $\mathfrak l \subset \cO_{K_0}$ be a prime 
ideal above $\ell$. Let $\lO_{0, \mathfrak l} = \cO_{K_0, \mathfrak l}$ be the completion at $\mathfrak l$ of $\cO_{K_0}$. Let $\lO_{\mathfrak l}(\mathscr A) = \cO(\mathscr A) \otimes_{\cO_{K_0}} \lO_{0, \mathfrak l}$. 
Consider the suborder $\lO_{0, \mathfrak l}[\pi, \pi^{\dagger}]$ of the maximal local (at $\mathfrak l$) order 
$\lO_{\mathfrak l} = \cO_{K} \otimes_{\cO_{K_0}} \lO_{0, \mathfrak l}$. Now write 
$$
\lO_{0, \mathfrak l}[\pi, \pi^{\dagger}] = \lO_{0, \mathfrak l} + \gamma_{\mathfrak l} \lO_{0, \mathfrak l}, \qquad \text{and} \qquad \lO_{\mathfrak l} = \lO_{0, \mathfrak l} +
{\gamma_{\mathfrak l} / \varpi^{f_{\mathfrak l}}} \lO_{0, \mathfrak l}, 
$$
for some endomorphism $\gamma$. Here, $\varpi$ is a uniformizer for the local order $\lO_{0, \mathfrak l}$ and $f_\mathfrak l \geq 0$ is some integer. To define a similar obstruction to $N_0(\mathscr A, \epsilon)$, but at $\mathfrak l$, let
$$
N_\mathfrak l(\mathscr A) = f_{\mathfrak l} - \max \{\delta \colon \gamma (\mathscr A[\mathfrak l^\delta]) = 0\}.  
$$
To compute 
these obstructions, we compute $\gamma$ on the $\mathfrak l$-power torsion of $\mathscr A$. 
The idea is similar to Algorithm~\ref{alg:surf-real}, except that in the split case, one must test the obstructions $N_\mathfrak l(\mathscr A, \epsilon)$ for both prime ideals $\mathfrak l \subset \cO_{K_0}$ above $\ell$ at the same time. 
We now show that one can reach the maximal possible ``reachable" (in the sense of Theorem~\ref{thm:ellellsurface}) local order at $\ell$ starting from $\mathscr A$ and using only $(\ell, \ell)$-isogenies. When $\ell$ is either inert or ramified in $K_0$, there is only one obstruction $N_{\mathfrak l}(\mathscr A)$, and one can ensure that it decreases at each step via the obvious modification of Algorithm~\ref{alg:surf-real}. 

Suppose now that $\ell \cO_{K_0} = \mathfrak l_1 \mathfrak l_2$ is split. 
Let $\mathfrak f = \mathfrak l_1^{i_1} \mathfrak l_2^{i_2}$ be the conductor of 
$\mathscr A$ and suppose, without loss of generality, that $i_1 \geq i_2$. To first ensure that one can reach an abelian surface $\mathscr A$ for which $0 \leq i_1 - i_2 \leq 1$, we relate the conductor $\mathfrak f$ to the two obstructions at $\mathfrak l_1$ and $\mathfrak l_2$. 

\begin{lemma}
Let $\mathscr A$ be an abelian surface with maximal local real endomorphism ring at $\ell$ and let 
$\lO(\mathscr A) = \lO_0 + \mathfrak f \lO_K$ where $\mathfrak f$ is the conductor. Then 
$$
v_{\mathfrak l_1}(\mathfrak f) = N_{\mathfrak l_1}(\mathscr A) \qquad \text{and} \qquad 
v_{\mathfrak l_2}(\mathfrak f) = N_{\mathfrak l_2}(\mathscr A). 
$$
\end{lemma}

\begin{proof}
The proof is the same as the one of Lemma~\ref{lem:obstr-cond-real}. 
\end{proof}

Using the lemma, and assuming $N_{\mathfrak l_1}(\mathscr A) - N_{\mathfrak l_2}(\mathscr A) > 1$, one repeatedly looks for an $(\ell, \ell)$-isogeny at each step that will decrease  $N_{\mathfrak l_1}(\mathscr A)$ by 1 and increase $N_{\mathfrak l_2}(\mathscr A)$ by 1. Such an isogeny exists by Proposition~\ref{prop:ellellIncreasingLocalOrder}(iii). One repeats this process until 
$$
0 \leq N_{\mathfrak l_1}(\mathscr A) - N_{\mathfrak l_2}(\mathscr A) \leq 1. 
$$ 
If at this stage $N_{\mathfrak l_2}(\mathscr A) > 0$, this means that $\ell \mid {\mathfrak f}$ and hence, by Proposition~\ref{prop:ellellIncreasingLocalOrder}(i), there exists 
a unique $(\ell, \ell)$-isogeny decreasing both obstructions. One searches for that $(\ell, \ell)$-isogeny by testing whether the two obstructions decrease simultaneously, and repeats until $N_{\mathfrak l_2}(\mathscr A) = 0$.  

If $N_{\mathfrak l_1}(\mathscr A) = 0$, then the maximal local order at $\ell$ has been reached. If $N_{\mathfrak l_1}(\mathscr A) = 1$ then Proposition~\ref{prop:ellellIncreasingLocalOrder}(iii) implies that, if $\mathfrak l_2$ is not inert in $K$, then 
there exists an $(\ell, \ell)$-isogeny that decreases $N_{\mathfrak l_1}(\mathscr A)$ to $0$ and keeps $N_{\mathfrak l_2}(\mathscr A)$ at zero.

\subsubsection{Final step via a cyclic isogeny}
In the exceptional cases of Theorem~\ref{thm:ellellsurface}, it may happen that one needs to do an extra step via a cyclic isogeny to reach maximal local endomorphism ring at $\ell$. Whenever this cyclic $\mathfrak l$-isogeny is computable via the algorithm of \cite{dudeanu-jetchev-robert}, one can 
always reach maximal local endomorphism ring at $\ell$. But $\mathfrak l$-isogenies are computable if and only if $\mathfrak l$ is trivial in the narrow class group of $K_0$. We thus distinguish the following two cases: 

\begin{enumerate}
\item If $\mathfrak l$-isogenies are computable by \cite{dudeanu-jetchev-robert} then one can always reach maximal local  endomorphism ring at $\ell$.
\item If $\mathfrak l$-isogenies are not computable by \cite{dudeanu-jetchev-robert}, one can only use $(\ell,\ell)$-isogenies, so Theorem~\ref{thm:ellellsurface} tells us what the largest order that we can reach is.
\end{enumerate}

\section*{Acknowledgements}
The first and second authors were supported by the Swiss National Science Foundation. The third author was supported by the Swiss National Science Foundation
under grant number 200021-156420.

\bibliographystyle{amsalpha}
\bibliography{biblio}

\newcommand{\etalchar}[1]{$^{#1}$}
\providecommand{\bysame}{\leavevmode\hbox to3em{\hrulefill}\thinspace}
\providecommand{\MR}{\relax\ifhmode\unskip\space\fi MR }
\providecommand{\MRhref}[2]{%
  \href{http://www.ams.org/mathscinet-getitem?mr=#1}{#2}
}
\providecommand{\href}[2]{#2}
\begin{thebibliography}{GHK{\etalchar{+}}06}

\bibitem[Bas63]{Bass63}
H.~Bass, \emph{On the ubiquity of {G}orenstein rings}, Mathematische
  Zeitschrift \textbf{82} (1963), no.~1, 8--28.

\bibitem[BGL11]{broker-gruenewald-lauter}
R.~Br{\"o}ker, D.~Gruenewald, and K.~Lauter, \emph{Explicit {CM} theory for
  level 2-structures on abelian surfaces}, Algebra Number Theory \textbf{5}
  (2011), no.~4, 495--528.

\bibitem[Bis15]{bisson}
G.~Bisson, \emph{Computing endomorphism rings of abelian varieties of dimension
  two}, Math. Comp. \textbf{84} (2015), no.~294, 1977--1989.

\bibitem[BL94]{BuchmannLenstra}
J.~Buchmann and H.~W. Lenstra, \emph{Approximating rings of integers in number
  fields}, J. Th{\'e}or. Nombres Bordeaux \textbf{6} (1994), no.~2, 221--260.
  \MR{1360644}

\bibitem[BL04]{BL04}
C.~Birkenhake and H.~Lange, \emph{Complex abelian varieties}, Die Grundlehren
  der mathematischen Wissenschaften in Einzeldarstellungen, Springer, 2004.

\bibitem[BLS12]{broker-lauter-sutherland}
R.~Br{\"o}ker, K.~Lauter, and A.~Sutherland, \emph{Modular polynomials via
  isogeny volcanoes}, Math. Comp. \textbf{81} (2012), no.~278, 1201--1231.

\bibitem[CKL08]{carls-kohel-lubicz}
R.~Carls, D.~Kohel, and D.~Lubicz, \emph{Higher-dimensional 3-adic {CM}
  construction}, J. Algebra \textbf{319} (2008), no.~3, 971--1006.

\bibitem[CL09]{carls-lubicz}
R.~Carls and D.~Lubicz, \emph{A {$p$}-adic quasi-quadratic time point counting
  algorithm}, Int. Math. Res. Not. IMRN (2009), no.~4, 698--735.

\bibitem[CQ05]{cardona-quer}
G.~Cardona and J.~Quer, \emph{Field of moduli and field of definition for
  curves of genus 2}, Computational aspects of algebraic curves, Lecture Notes
  Ser. Comput., vol.~13, World Sci. Publ., Hackensack, NJ, 2005, pp.~71--83.

\bibitem[CR15]{cosset-robert}
R.~Cosset and D.~Robert, \emph{Computing {$(\ell,\ell)$}-isogenies in
  polynomial time on {J}acobians of genus {$2$} curves}, Math. Comp.
  \textbf{84} (2015), no.~294, 1953--1975.

\bibitem[CV04]{Cornut04}
C.~Cornut and V.~Vatsal, \emph{Nontriviality of {R}ankin-{S}elberg
  {L}-functions and {C}{M} points}, Tech. report, 2004.

\bibitem[DJR16]{dudeanu-jetchev-robert}
A.~Dudeanu, D.~Jetchev, and D.~Robert, \emph{Cyclic isogenies for abelian
  varieties with real multiplication}, preprint (2016).

\bibitem[EL10]{eisentraeger-lauter}
K.~Eisentr{\"a}ger and K.~Lauter, \emph{A {CRT} algorithm for constructing
  genus 2 curves over finite fields}, Arithmetics, geometry, and coding theory
  ({AGCT} 2005), S\'emin. Congr., vol.~21, Soc. Math. France, Paris, 2010,
  pp.~161--176.

\bibitem[ES10]{enge-sutherland}
A.~Enge and A.~Sutherland, \emph{Class invariants by the {CRT} method},
  Algorithmic number theory, Lecture Notes in Comput. Sci., vol. 6197,
  Springer, Berlin, 2010, pp.~142--156.

\bibitem[FL08]{freeman-lauter}
D.~Freeman and K.~Lauter, \emph{Computing endomorphism rings of {J}acobians of
  genus 2 curves over finite fields}, Algebraic geometry and its applications,
  Ser. Number Theory Appl., vol.~5, World Sci. Publ., Hackensack, NJ, 2008,
  pp.~29--66.

\bibitem[FM02]{fouquet-morain}
M.~Fouquet and F.~Morain, \emph{Isogeny volcanoes and the {SEA} algorithm},
  Algorithmic number theory ({S}ydney, 2002), Lecture Notes in Comput. Sci.,
  vol. 2369, Springer, Berlin, 2002, pp.~276--291.

\bibitem[Gel75]{Gelbart}
Stephen~S. Gelbart, \emph{Automorphic forms on ad{\`e}le groups}, Princeton
  University Press, Princeton, N.J.; University of Tokyo Press, Tokyo, 1975,
  Annals of Mathematics Studies, No. 83. \MR{0379375}

\bibitem[GHK{\etalchar{+}}06]{ghkr:2adic}
P.~Gaudry, T.~Houtmann, D.~Kohel, C.~Ritzenthaler, and A.~Weng, \emph{The
  2-adic {CM} method for genus 2 curves with application to cryptography},
  Advances in cryptology---{ASIACRYPT} 2006, Lecture Notes in Comput. Sci.,
  vol. 4284, Springer, Berlin, 2006, pp.~114--129.

\bibitem[GL09]{GorenLauter09}
E.~Goren and K.~Lauter, \emph{The distance between superspecial abelian
  varieties with real multiplication}, Journal of Number Theory \textbf{129}
  (2009), no.~6, 1562 -- 1578.

\bibitem[IT14]{ionica-thome}
S.~Ionica and E.~Thom\'e, \emph{Isogeny graphs with maximal real
  multiplication}, \texttt{https://arxiv.org/pdf/1407.6672v1.pdf} (accessed in
  Sept. 2016), 2014.

\bibitem[JMV05]{jmv:asiacrypt}
D.~Jao, S.~D. Miller, and R.~Venkatesan, \emph{Do all elliptic curves of the
  same order have the same difficulty of discrete log?}, Advances in
  cryptology---{ASIACRYPT} 2005, Lecture Notes in Comput. Sci., vol. 3788,
  Springer, Berlin, 2005, pp.~21--40.

\bibitem[JMV09]{jmv:jnt}
\bysame, \emph{Expander graphs based on {GRH} with an application to elliptic
  curve cryptography}, J. Number Theory \textbf{129} (2009), no.~6, 1491--1504.

\bibitem[JW15]{JW15}
D.~Jetchev and B.~Wesolowski, \emph{On graphs of isogenies of principally
  polarizable abelian surfaces and the discrete logarithm problem}, CoRR
  \textbf{abs/1506.00522} (2015).

\bibitem[Koh96]{kohel:thesis}
D.~Kohel, \emph{Endomorphism rings of elliptic curves over finite fields},
  ProQuest LLC, Ann Arbor, MI, 1996, Thesis (Ph.D.)--University of California,
  Berkeley.

\bibitem[LR12a]{lauter-robert}
K.~Lauter and D.~Robert, \emph{Improved crt algorithm for class polynomials in
  genus 2}, IACR Cryptology ePrint Archive \textbf{2012} (2012), 443.

\bibitem[LR12b]{lubicz-robert}
D.~Lubicz and D.~Robert, \emph{Computing isogenies between abelian varieties},
  Compos. Math. \textbf{148} (2012), no.~5, 1483--1515.

\bibitem[Mes91]{mestre:genus2}
J.-F. Mestre, \emph{Construction de courbes de genre {$2$} \`a partir de leurs
  modules}, Effective methods in algebraic geometry ({C}astiglioncello, 1990),
  Progr. Math., vol.~94, Birkh\"auser Boston, Boston, MA, 1991, pp.~313--334.

\bibitem[Mil86]{MilneAV}
J.~S. Milne, \emph{Abelian varieties}, Arithmetic geometry ({S}torrs, {C}onn.,
  1984), Springer, New York, 1986, pp.~103--150. \MR{861974}

\bibitem[Mum66]{mumford:eq1}
D.~Mumford, \emph{On the equations defining abelian varieties. {I}}, Invent.
  Math. \textbf{1} (1966), 287--354.

\bibitem[NS99]{Neukirch99}
J.~Neukirch and N.~Schappacher, \emph{Algebraic number theory}, Grundlehren der
  mathematischen Wissenschaften, Springer, Berlin, New York, Barcelona, 1999.

\bibitem[Rob10a]{drobert:thesis}
D.~Robert, \emph{Fonctions th\^eta et applications \`a la cryptologie}, PhD
  thesis, Universit\'e Henri Poincar\'e - Nancy I (2010).

\bibitem[Rob10b]{robert}
\bysame, \emph{Theta functions and applications in cryptography}, Ph.D. thesis,
  Loria, Nancy, 2010.

\bibitem[ST61]{taniyama-shimura}
G.~Shimura and Y.~Taniyama, \emph{Complex multiplication of abelian varieties
  and its applications to number theory}, Publications of the Mathematical
  Society of Japan, vol.~6, The Mathematical Society of Japan, Tokyo, 1961.

\bibitem[Str10]{Streng10}
Marco Streng, \emph{Complex multiplication of abelian surfaces}, Ph.D. thesis,
  Universiteit Leiden, 2010.

\bibitem[Sut11]{sutherland:crt}
A.~Sutherland, \emph{Computing {H}ilbert class polynomials with the {C}hinese
  remainder theorem}, Math. Comp. \textbf{80} (2011), no.~273, 501--538.

\bibitem[Sut12]{sutherland}
\bysame, \emph{Accelerating the {CM} method}, LMS J. Comput. Math. \textbf{15}
  (2012), 172--204.

\bibitem[Tat66]{Tate1966}
J.~Tate, \emph{Endomorphisms of abelian varieties over finite fields},
  Inventiones mathematicae \textbf{2} (1966), no.~2, 134--144.

\bibitem[V{\'e}l71]{velu}
J.~V{\'e}lu, \emph{Isog\'enies entre courbes elliptiques}, C. R. Acad. Sci.
  Paris S\'er. A-B \textbf{273} (1971), A238--A241.

\bibitem[vW99]{vanwamelen:genus2}
P.~van Wamelen, \emph{Examples of genus two {CM} curves defined over the
  rationals}, Math. Comp. \textbf{68} (1999), no.~225, 307--320.

\bibitem[Wat69]{Waterhouse1969}
W.~Waterhouse, \emph{Abelian varieties over finite fields}, Annales
  scientifiques de l'\'Ecole Normale Sup\'erieure \textbf{2} (1969), no.~4,
  521--560 (eng).

\bibitem[Wen03]{weng:genus2}
A.~Weng, \emph{Constructing hyperelliptic curves of genus 2 suitable for
  cryptography}, Math. Comp. \textbf{72} (2003), no.~241, 435--458
  (electronic).

\end{thebibliography}

\end{document}